\documentclass[10pt,reqno]{amsart}

\usepackage[a4paper,left=35mm,right=35mm,top=30mm,bottom=30mm,marginpar=25mm]{geometry}

\usepackage{amsmath}
\usepackage{amssymb}
\usepackage{amsthm}
\usepackage[utf8]{inputenc}
\usepackage{eurosym}
\usepackage{graphicx}
\usepackage{hyperref}
\usepackage{dsfont}
\usepackage{dsfont}
\usepackage{xcolor,colortbl}
\usepackage{comment}

\allowdisplaybreaks

\usepackage{hyperref}
\usepackage{ifthen}

\makeindex

\newcommand{\Rr}{{\mathbb{R}}}

\newcommand{\Nn}{{\mathbb{N}}}

\newcommand{\Ee}{{\mathds{E}}}%notice it is not mathbb

\newcommand{\Ff}{{\mathcal{F}}}

\newcommand{\Pp}{{\mathcal{P}}}

\newcommand{\bx}{{\bf x}}

\newcommand{\bX}{{\bf X}}

\newcommand{\bZ}{{\bf Z}}

\newcommand{\bP}{{\bf P}}

\newcommand{\bv}{{\bf v}}

\newcommand{\bw}{{\bf w}}

\newcommand{\tX}{\tilde{X}}
\newcommand{\tZ}{\tilde{Z}}
\newcommand{\tP}{\tilde{P}}
\newcommand{\tH}{\tilde{H}}

\def\leq{\leqslant}
\def\geq{\geqslant}

\numberwithin{equation}{section}

\newtheoremstyle{thmlemcorr}{10pt}{10pt}{\itshape}{}{\bfseries}{.}{10pt}{{\thmname{#1}\thmnumber{
#2}\thmnote{ (#3)}}}
\newtheoremstyle{thmlemcorr*}{10pt}{10pt}{\itshape}{}{\bfseries}{.}\newline{{\thmname{#1}\thmnumber{
#2}\thmnote{ (#3)}}}
\newtheoremstyle{defi}{10pt}{10pt}{\itshape}{}{\bfseries}{.}{10pt}{{\thmname{#1}\thmnumber{
#2}\thmnote{ (#3)}}}
\newtheoremstyle{remexample}{10pt}{10pt}{}{}{\bfseries}{.}{10pt}{{\thmname{#1}\thmnumber{
#2}\thmnote{ (#3)}}}
\newtheoremstyle{ass}{10pt}{10pt}{}{}{\bfseries}{.}{10pt}{{\thmname{#1}\thmnumber{
A#2}\thmnote{ (#3)}}}

\theoremstyle{thmlemcorr}
\newtheorem{theorem}{Theorem}
\numberwithin{theorem}{section}
\newtheorem{lemma}[theorem]{Lemma}

\newtheorem{proposition}[theorem]{Proposition}

\theoremstyle{thmlemcorr*}
\newtheorem{theorem*}{Theorem}
\newtheorem{lemma*}[theorem]{Lemma}
\newtheorem{corollary*}[theorem]{Corollary}
\newtheorem{proposition*}[theorem]{Proposition}
\newtheorem{problem*}[theorem]{Problem}
\newtheorem{conjecture*}[theorem]{Conjecture}

\theoremstyle{defi}

\newtheorem{hyp}{Assumption}

\newtheorem{problem}{Problem}

\theoremstyle{remexample}
\newtheorem{remark}[theorem]{Remark}

\theoremstyle{ass}

\usepackage{tikz}
\usetikzlibrary{arrows}
\usetikzlibrary{positioning}
\usepackage{bm}

\begin{document}

\title[A random-supply MFG price model]{A random-supply Mean Field Game price model}

\author{Diogo Gomes}\thanks{King Abdullah University of Science and Technology (KAUST), CEMSE Division, Thuwal 23955-6900.
Saudi Arabia.
e-mail: diogo.gomes@kaust.edu.sa.}%
\author{Julian Gutierrez}\thanks{King Abdullah University of Science and Technology (KAUST), CEMSE Division, Thuwal 23955-6900.
Saudi Arabia.
e-mail: julian.gutierrezpineda@kaust.edu.sa.}%
\author{Ricardo Ribeiro}\thanks{King Abdullah University of Science and Technology (KAUST), CEMSE Division, Thuwal 23955-6900.
Saudi Arabia.
e-mail: ricardo.ribeiro@kaust.edu.sa.}

%\affil[1]{King Abdullah University of Science and Technology - KAUST}%
%\affil[2]{Affiliation not available}%
%\affil[3]{Kwame Nkrumah University of Science and Technology}%

%\author{Diogo A.
%  Gomes}
%\address[D.~A.~Gomes]{
%        King Abdullah University of Science and Technology (KAUST), CEMSE Division, Thuwal 23955-6900.
% Saudi Arabia, and  
%        KAUST SRI, Center for Uncertainty Quantification in Computational Science and Engineering.}
%\email{diogo.gomes@kaust.edu.sa}
%\author{John Doe}
%\address[J.~Doe]{
%        The University Address}
%\email{doe@doe.tv}

%\subjclass[2010]{
%        35J47, %Second order elliptic systems
%        35A01} %Existence problems: global existence, local existence, non-existence

\thanks{
      The authors were partially supported by KAUST baseline funds and 
 KAUST OSR-CRG2017-3452.
}
\date{\today}

\begin{abstract}
We consider a market where a finite number of players trade an asset whose supply is a stochastic process. The price formation problem consists of finding a price process that ensures that when agents act optimally to minimize their trading costs, the market clears, and supply meets demand. This problem arises in market economies, including electricity generation from renewable sources in smart grids. Our model includes noise on the supply side, which is counterbalanced on the consumption side by storing energy or reducing the demand according to a dynamic price process. By solving a constrained minimization problem, we prove that the Lagrange multiplier corresponding to the market-clearing condition defines the solution of the price formation problem. For the linear-quadratic structure, we characterize the price process of a continuum population using optimal control techniques. We include numerical schemes for the price computation in the finite and infinite games, and we illustrate the model using real data.
\end{abstract}

\maketitle

\section{introduction}
Mean-field game theory (MFG) is an approach to study the evolution of a population of competitive rational players.
Each player solves an optimal control problem that depends on statistical features of the population rather than one-to-one interactions.
The statistical features inform the objective of each agent, determining their dynamics. Adopting a MFG approach, the authors in \cite{gomes2018mean} addressed a deterministic price formation model with a market-clearing condition in which the objectives of a continuum of agents are coupled to the price.
In this paper, we study a price formation model where $N$ agents interact in a market via the price, $\varpi$, of the commodity they trade and whose supply is random.
The agents meet a balance condition that guarantees the supply, $Q$, of the commodity equals its demand.
The novelty of our model consists of considering a random supply, such as electricity generation from sustainable sources.

The randomness in price formation has potential applications in renewable energy production on smart grids.
Small devices in the grid can store energy that can be sold back to the grid. Changes in weather conditions and network load cause fluctuations in the available supply. Because the agents can sell the surplus of power, they can benefit from load-adaptive pricing (\cite{MK14},  \cite{ATM19}).
%One approach to load-adaptive pricing is price formation.

To model price formation, there are two different approaches. One approach assumes that the price is a function of the variables in the model.  In this setting, \cite{BS10} compared different pricing policies under partially incomplete, complete, and totally complete information. Their model consisted of a reverse Stackelberg game with non-linear dependence on the price, and price formation is obtained by optimizing the producer's revenue. The work \cite{TBD20} presented a Cournot model  that specified the log-price dynamics, including a Brownian motion and a jump process as common noise. In \cite{ATM19}, the spot price is given as a strictly increasing function of the exogenous demand function and the mean energy trading rates. The optimal trading rates were determined by solving a forward-backward system that characterizes the mean-field equilibria. The same authors extended this model in \cite{alasseur2021mfg} 
to include penalty terms at random jump times in the state variables. The spot price is an inverse demand function of the expected consumption. They used forward-backward and Riccati equations with jumps to characterize the mean-field equilibrium. The work \cite{AidDumitrescuTankov2021} considered a MFG of optimal stopping to model the switch between traditional and renewable means of energy production. They considered a MFG where the market price couples the agents dynamics, and it is prescribed as a function of a price cap, the exogenous demand, and the supply of both the conventional and the renewable means of production. In their model, the market price is prescribed as a function of a price cap, the exogenous demand, and the supply of both the conventional and the renewable means of production. Recent works have examined the case of intraday electricity markets. \cite{FTT20} studied a linear-quadratic model in the presence of a major player. They distinguished the fundamental price (with no market impact) from the market price. The market price has an explicit form in terms of  the average position of the agents, the position of the major agent, and the fundamental price, which is an exogenous variable for the model. The same approach was taken in their consecutive work \cite{feron2021price}, where the market price depends on the fundamental price and the average position of the agents. They derived a MFG  formulation using conditional expectations w.r.t. the common noise and presented a convergence result between the finite model equilibrium to the mean-field model equilibrium as the number of players goes to infinity. They illustrated their results for the EPEX intraday electricity market. 

Our work follows a second approach, which was first introduced in \cite{Gomes20164693} and \cite{gomes2018mean}. In this approach, the price is unknown and determined by a balance condition. For instance, \cite{BS02} proposed a Stackelberg game for revenue-maximization with a linear dependence on the price. The price is obtained using the first-order conditions for the optimization problem. A model for Solar Renewable Energy Certificate Markets (SREC) was presented in \cite{JSF20}, where the supply of the energy being priced is controlled. They obtained the SREC price using a market clearing condition and a first-order optimality condition for the optimal planned generation and energy trading. \cite{FT20} obtained the equilibrium price using a market clearing condition and a forward-backward system of the McKean-Vlasov type characterizing the optimal trading rate for the agents. The same authors studied in \cite{fujii2021equilibrium} a further extension that considers a Major player in the market. The market price is characterized by the solution to a forward-backward stochastic differential equation (SDE) system and a market-clearing condition. In \cite{aid2020equilibrium}, the authors presented a model of $N$ agents with demand forecasts subjected to common noise. In their model, agents meet the demand by selecting controls on their production and trading rate, satisfying an equilibrium condition. The price is obtained using the existence result for a forward-backward coupled system. \cite{Evangelista2022} studied the convergence of a finite-population game to a MFG for a model where traders control their turnover rates with noise in the inventory. They considered a market clearing condition between the aggregated inventory and the supply. The price was obtained by characterizing the Nash equilibrium of the finite-population game using a forward-backward SDE. They illustrated their results using real high-frequency data.

Because the works \cite{JSF20}, \cite{FT20}, and \cite{aid2020equilibrium} deal with a model similar to the one we consider, let us emphasize the novelty in our work. The model in \cite{JSF20} is specialized in the SREC markets, which provides further structure to the model formulation, such as a quadratic cost structure. They used a forward-backward system and variational techniques to formulate a fixed-point problem to prove the existence of a mean-field distribution, from which they get the price. In contrast, we deal with a general convex cost, illustrate our results for the quadratic case, and prove existence using a variational approach. In \cite{FT20}, the authors approximated the equilibrium price by conditioning a stochastic process to the filtration induced by the common noise. They showed that this approximation satisfies the market clearing condition when the number of players increases. In distinction, we obtain a price for which the balance condition for the $N$-players hold. Lastly, \cite{aid2020equilibrium} derived the price equilibrium using the existence and uniqueness result for a forward-backward coupled system. In contrast, our existence results rely on the calculus of variations approach, whereas we derive forward-backward systems as necessary conditions for such existence. These conditions allow us to identify the price as the Lagrange multiplier for a $N$-agent minimization problem with constraints in which the price is no longer present.

The case of a finite number of players with a deterministic supply was addressed in \cite{SummerCamp2019}, where only existence and uniqueness were proved, and no numerical approximation scheme was considered. The model we present here generalizes the deterministic supply case. In \cite{GoGuRi2021}, we addressed the stochastic supply case from the optimal control perspective, and we provided numerical results for a quadratic Lagrangian depending on the trading rate only. The main contribution of this paper is the proof of existence and uniqueness of solutions for the price formation model with a finite number of players in the stochastic case under a general cost function. We adopt a variational approach to obtain our results, and we elaborate on the numerical approximation of solutions.

Next, we introduce our model. Let $T>0$ be the time horizon.
In the following, we fix  a complete filtered probability space $(\Omega,\Ff,\mathds{F},\mathds{P})$; that is, $\mathds{F}=(\Ff_t)_{0\leq t \leq T}$ is the standard filtration generated by $t\mapsto W_t$, a Brownian motion in $\Rr$ (see \cite{MR2001996}, Definition 3.1.3, and \cite{BackwardSDE1997}, Section 2, for additional details).
Here, $W$ plays the role of the common noise in the sense that the supply follows the stochastic differential equation (SDE)
\begin{equation}\label{eq: supply dynamics intro}
dQ_t=b^S(Q_t,t)dt+\sigma^S(Q_t,t)dW_t.
\end{equation}

%(\cite{karatzas1991brownian}, Corollary 3.4).
%, that is, $\Ff_s\subset \Ff_t \subset \Ff$ for $0\leq s < t\leq T$ and $W_t$ is adapted to $\Ff_t$, which is a complete $\sigma$-algebra of $\Omega$ 

In our model, the agent's interaction determines the market equilibrium price of the commodity.
All of this commodity produced is consumed entirely.
Let $N$ be the number of agents and let the state variable $X^i_t$ account for the quantity of the commodity held by agent $i$ at time $t$.
Each agent controls its trading rate according to
\begin{equation}\label{eq: Agent dynamics}
dX_t=v_t dt, ~ t\in[0,T],
\end{equation}
where $v:[0,T]\times \Omega \to \Rr$, the control variable, is progressively measurable with respect to $\mathds{F}$.
The optimization problem we consider reads:
\begin{problem}\label{problem: problem}
Let $N$ be the number of agents.
Let the supply, $Q$, be a stochastic process adapted to $\mathds{F}$ solving \eqref{eq: supply dynamics intro}.
Let $L \in C^1(\Rr^2;\Rr)$ be a non-negative Lagrangian, and $\Psi \in C^1(\Rr)$ be a non-negative terminal cost.
Assume that at time $t=0$, each agent $i$ owns a quantity $x_0^i\in\Rr$ of the commodity.

Find a price process $\varpi$ and control processes $v^i$, all adapted to $\mathds{F}$, such that for each $i$, with $1\leq i \leq  N$, $X^i$ solves \eqref{eq: Agent dynamics} with the initial condition $X^i_0=x^i_0$, and minimizes the cost functional
\begin{equation}\label{eq: Functional per agent}
\Ee\left[ \int_0^T L(X^i_t,v^i_t) + \varpi_t v_t^i~ dt + \Psi(X^i_T) \right],
\end{equation}
subject to the balance condition
\begin{equation}\label{eq:Balance condition}
\frac{1}{N} \sum_{i=1}^N v_t^i = Q_t, \quad \mbox{\emph{ for }} 0\leq t \leq T.
\end{equation}
\end{problem}

The functional \eqref{eq: Functional per agent} represents the expected cost for a representative agent on $[0,T]$.
This cost consists of three parts: the trading at the current price through the linear term $\varpi v$, the charges related to storage or market impact encoded in $L$, and the terminal cost; the terminal cost reflects the preferences of players at the terminal time.
The balance condition \eqref{eq:Balance condition} guarantees demand consumes all supply.
For this problem, we obtain the following result:

\begin{theorem}\label{thm:MainThm}
Let\footnote{See Section \ref{sec:AssumptionsNotation} for notation and assumptions.} $Q\in \mathds{H}_{\mathds{F}}$ and suppose that Assumptions \ref{hyp: L jconvex}-%\ref{hyp: Psi convex}, \ref{hyp: Psi growth}, \ref{hyp: L Lx Lv},  and 
\ref{hyp: Coercivity} hold.
Then, there exists control processes ${v^*}^i$, for $1\leq i \leq N$, and a price process $\varpi$ that solve Problem \ref{problem: problem}.
Furthermore, under Assumption \ref{hyp: L uniformly convex}, the price $\varpi$ and the control processes ${v^*}^i$, for $1\leq i \leq N$, solving Problem \ref{problem: problem}, are unique.
\end{theorem}

We prove this theorem in Section \ref{sec: NAgent problem}, where we formulate a problem independent of the price, but the constraint imposed by the balance condition is still present. Existence for this problem is obtained by the direct method in the calculus of variations, and we obtain a forward-backward characterization of optimizers, which allows identifying the price as the Lagrange multiplier corresponding to the balance constraint. 

The outline of the paper is as follows: In Section \ref{sec:AssumptionsNotation}, we introduce the main assumptions for the model as well as the notation for the function spaces. In Section \ref{sec: single agent problem}, we study the optimization problem that a representative agent solves under the assumption that the price is known, which corresponds to the optimization problem that all agents solve simultaneously in the the $N$-agent problem. Using the representative agent result, we prove the existence of a solution to the $N$ agent price formation problem, Problem \ref{problem: problem}, in Section \ref{sec: NAgent problem}. We specialize our results for a linear-quadratic structure of the model in Section \ref{Sec: LQ model}. Using optimal control techniques and an extended-state space approach, we obtain semi-explicit expressions for the price with finite $N$ agents and infinite agents. We discuss the convergence as $N\to \infty$ of the former to the latter. The general case is beyond the scope of this paper. The numerical computation of the price is discussed in Section \ref{sec:numerical}, where we present numerical results for a generic model of Section \ref{Sec: LQ model} and a calibrated model based on real data from the electricity grid in Spain.

\section{Assumptions and notation}\label{sec:AssumptionsNotation}
We consider natural assumptions in the context of the calculus of variations (see \cite{Dacorogna2}).
The following conditions are used to prove the existence of minimizers of \eqref{eq: Single agent min problem}, \eqref{eq: Nagent min problem}, and \eqref{eq: reduced Nagent min problem}.
In the following, we suppose the Lagrangian $L$ is non-negative.
%\begin{hyp}\label{hyp: L strictly convex}
%$L$ is strictly convex in the second component, that is, for all $x\in \Rr$
%\begin{equation}\label{hyp eq: L strictly convex}
%L(x,\lambda v + (1-\lambda) \tv) < \lambda L(x,v)+(1-\lambda) L(x,\tv) \quad \mbox{ for all } \lambda \in (0,1).
%\end{equation}
%\end{hyp}
\begin{hyp}\label{hyp: L jconvex}
The Lagrangian $L \in C^1(\Rr^2;\Rr^+\cup\{0\})$ is convex in $(x,v)$; that is, $(x,v)\mapsto L(x,v)$ is convex.
\end{hyp}

\begin{hyp}\label{hyp: Psi convex}
The terminal cost $\Psi \in C^1(\Rr;\Rr^+\cup\{0\})$ is convex.
\end{hyp}

Because we consider integrals w.r.t.
measure spaces, we require compositions of processes with functions to remain in the same class where the process is taken.
The following growth conditions guarantee this.
\begin{hyp}\label{hyp: Psi growth}
$\Psi \in C^1(\Rr;\Rr^+\cup\{0\})$ satisfies, for some $C>0$,
\[
	\Psi(x) \leq C(1+|x|^2), \quad \mbox{for all } x\in\Rr.
\]
Moreover, its derivative, which we denote by $\Psi'$, satisfies, for some $\tilde{C}>0$,
\[
	|\Psi'(x)| \leq \tilde{C} (1+|x|).
\]
\end{hyp}

\begin{hyp}\label{hyp: L Lx Lv}
$L \in C^1(\Rr^2;\Rr^+\cup\{0\})$, and there exists $ \tilde{\beta},C>0$ such that
\begin{align*}
	L(x,v) &\leq \tilde{\beta} (1+|v|^2), \quad \mbox{ for all }x \in \Rr,
	\\
	|L_x(x,v)|,|L_v(x,v)| &\leq C (1+|v|), \quad \mbox{for all } (x,v)\in\Rr^2.
\end{align*}
\end{hyp}

In convex optimization, a natural assumption to obtain the existence of minimizers is the coercivity condition.
\begin{hyp}\label{hyp: Coercivity}
(Coercivity) For some $\alpha>0$ and $\beta\geq 0$
\begin{equation*}
	\alpha |v|^2 - \beta \leq L(x,v), \quad \mbox{ for all } x,v \in \Rr.
\end{equation*}
\end{hyp}

To guarantee the uniqueness of minimizers, we consider next a strong form of convexity.
In turn, this assumption implies the coercivity condition (\cite{bauschke2017convex}, Corollary 11.17).

\begin{hyp}\label{hyp: L uniformly convex}
(Uniform convexity) For some $\theta>0$
\begin{equation*}
v \mapsto L(x,v)-\frac{\theta}{2} |v|^2 \quad \mbox{ is convex for all } x\in \Rr.
\end{equation*}
\end{hyp}

We introduce the Hamiltonian, $H$, the Legendre transform of $L$, by
\begin{equation}\label{def: Legendre trans}
	H(x,p)=\sup_{v\in\Rr}\left\{-pv-L(x,v)\right\}.
\end{equation}
Recall that when the map $v\mapsto L(x,v)$ is convex, $H(x,p)$ is well defined.
Furthermore, if $v\mapsto L(x,v)$ is strictly convex, $L\in C^2(\Rr^2;\Rr)$, and Assumption \ref{hyp: Coercivity} holds, there exists a unique value $v^*$ where the supremum is attained.
In addition,
\begin{equation}\label{eq: Aux Legendre T  v p}
	v^*=-H_p(x,p) ~ \mbox{if and only if} ~ p=-L(x,v^*), \mbox{ and hence }H(x,p)=-pv^*-L(x,v^*).
\end{equation}
See \cite{cannarsa}, Theorem A.
2.5, for the proof of the previous results.
For the Hamiltonian, we additionally require no more than linear growth of the gradient in the $p$ component, as we state next.

\begin{hyp}\label{hyp: H_p growth}
The Hamiltonian $H$ satisfies, for some $C>0$,
\[
	|H_p(x,p)| \leq C(1+|p|), \quad \mbox{for all } (x,p)\in\Rr^2.
\]
\end{hyp}

Now, we set up the notation.
Define the space $\mathds{H}_{\mathds{F}}$ as the set of processes $v:[0,T]\times \Omega \to \Rr$, that are measurable and adapted w.r.t.
$\mathds{F}$, and satisfy $\|v\|_{\mathds{H}_{\mathds{F}}}^2 < \infty$, where
\[
	\langle v , w \rangle_{\mathds{H}_{\mathds{F}}} := \Ee\left[ \int_0^T v_t w_t dt \right], \quad \|v\|_{\mathds{H}_{\mathds{F}}}^2:= \langle v , v \rangle_{\mathds{H}_{\mathds{F}}}.
\]
This expectation is w.r.t.
the measure induced by the Brownian motion.
$\mathds{H}_{\mathds{F}}$ is a Hilbert space (\cite{carmona2018probabilistic}, Remark 2.2.).
Given $v\in \mathds{H}_{\mathds{F}}$, the solution to \eqref{eq: Agent dynamics} with the initial condition $x_0\in \Rr$ is
\[
	X_t = x_0 + \int_0^t v_s ds.
\] 
Notice that $X \in \mathds{H}_{\mathds{F}}$ because $	\|X\|_{\mathds{H}_{\mathds{F}}}^2 \leq 2T|x_0|^2  + 2 T^2\|v\|_{\mathds{H}_{\mathds{F}}}^2$.
For our purposes, we consider trajectories with initial condition $x_0\in \Rr$.

For $N\in \Nn$, we define $\mathds{H}_{\mathds{F}}^N$, where $\bv = (v^1,\ldots,v^N)\in \mathds{H}_{\mathds{F}}^N$ provided $v^i \in \mathds{H}_{\mathds{F}}$, and
\[
	\langle \bv , \bw \rangle_{\mathds{H}_{\mathds{F}}^N} := \sum_{i=1}^N \langle v^i , w^i \rangle_{\mathds{H}_{\mathds{F}}}, \quad \|\bv\|_{\mathds{H}_{\mathds{F}}^N}^2:= \sum_{i=1}^N \|v^i\|_{\mathds{H}_{\mathds{F}}}^2.
\]

The analysis of Problem \ref{problem: problem} relies on the results for the optimization problem faced by a representative agent, which we consider in the next section.

\section{The optimization problem for a representative agent}\label{sec: single agent problem}
In this section, we assume that a price, $\varpi$, is given.
We derive a weak formulation for the Euler-Lagrange equation associated with the optimal control problem for a representative agent. We use this result in Section \ref{sec: NAgent problem} to study how the collective actions of the agents determine the price.

Let $x_0 \in \Rr$.
Given $v\in \mathds{H}_{\mathds{F}}$, consider the dynamics for the agent
\begin{equation}\label{eq: agent dynamics with initial condition} 
\begin{cases} dX_t=v_t dt ,~ t\in[0,T]\\ X_0=x_0.\end{cases}
\end{equation}
Given a price process $\varpi \in \mathds{H}_{\mathds{F}}$, the agent selects $v\in \mathds{H}_{\mathds{F}}$ aiming to reach
\begin{gather}\label{eq: Single agent min problem}
\inf_{v\in \mathds{H}_{\mathds{F}}} \Ee\left[ \int_0^T L(X_t,v_t) +\varpi_t  v_t~ dt + \Psi(X_T) \right]
\\
\mbox{subject to }  \quad X ~ \mbox{ solves } ~\eqref{eq: agent dynamics with initial condition}.
\nonumber
\end{gather}
Let 
\[
	I[v]:=\Ee\left[ \int_0^T L(X_t,v_t) +\varpi_t  v_t~ dt + \Psi(X_T) \right],
\]
where $X$ solves \eqref{eq: agent dynamics with initial condition} for $v$. In the following, we study the existence and uniqueness of solutions to \eqref{eq: Single agent min problem}. We adopt the direct method of the calculus of variations. Hence, we begin by proving that the functional $I[\cdot]$ is weakly lower semi-continuous.
\begin{proposition}\label{prop: wlsc of I}
Let $x_0\in\Rr$ and $\varpi \in \mathds{H}_{\mathds{F}}$.
Under Assumptions \ref{hyp: L jconvex}-%, \ref{hyp: Psi convex}, \ref{hyp: Psi growth}, and 
\ref{hyp: L Lx Lv}, the functional $I[\cdot]$ is weakly lower semi-continuous in $\mathds{H}_{\mathds{F}}$.
\end{proposition}
\begin{proof}
We will prove that $I[\cdot]$ is convex and lower semi-continuous, from which weak lower semi-continuity follows (\cite{kurdila2006convex} Theorem 7.2.5).
First, notice that, by Assumptions \ref{hyp: L jconvex} and \ref{hyp: Psi convex}, $I[\cdot]$ is convex.
To prove lower semi-continuity, let $(v^k)_{k\in\Nn}$ in $\mathds{H}_{\mathds{F}}$ be  such that $v^k$ converges to $v$.
Denote by $X^k$ and $X$ the solutions to \eqref{eq: agent dynamics with initial condition} with the controls $v^k$ and $v$, respectively.
Notice that, because the trajectories $X^k$ and $X$ have the same initial condition, we have 
\[
	\|X^k-X\|_{\mathds{H}_{\mathds{F}}}^2 \leq  2 T^2\|v^k-v\|_{\mathds{H}_{\mathds{F}}}^2.
\]
Therefore, $X^k$ converges to $X$.
The convexity in Assumptions \ref{hyp: L jconvex} and \ref{hyp: Psi convex} imply (\cite{bauschke2017convex}, Proposition 17.7)
\begin{equation}\label{eq: aux LowerSC of I 1}
	L_v(X_t,v_t)(v^k_t-v_t) + L(X_t,v_t) \leq L(X_t,v^k_t),
\end{equation}
\begin{equation}\label{eq: aux LowerSC of I 2}
	\Psi'(X_T)(X_T^k-X_T) + \Psi(X_T) \leq \Psi(X_T^k).
\end{equation}
Adding $L(X^k_t,v^k_t)-L(X_t,v^k_t)+\varpi_t v_t$ to both sides of \eqref{eq: aux LowerSC of I 1}, we get
\begin{align*}
	& L_v(X_t,v_t)(v^k_t-v_t) -L(X_t,v^k_t)+L(X^k_t,v^k_t) + L(X_t,v_t)+\varpi_t v_t
	\\
	& \leq \varpi_t (v_t-v^k_t) +L(X^k_t,v^k_t)+ \varpi_t v^k_t.
\end{align*}
Taking $\Ee[\int_0^T \cdot ~dt]$ in the previous inequality, $\Ee[\cdot]$ in \eqref{eq: aux LowerSC of I 2}, and adding both results, we obtain
\begin{align}\label{eq: aux LowerSC of I 3}
	&\langle L_v(X,v),v^k-v\rangle_{\mathds{H}_{\mathds{F}}} + \Ee\left[ \int_0^T L(X^k_t,v^k_t)-L(X_t,v^k_t)~ dt \right] + I[v] + \Ee\left[ \Psi'(X_T)(X^k_T-X_T)\right]\nonumber
	\\
	& \leq \langle \varpi, v-v^k\rangle_{\mathds{H}_{\mathds{F}}} +I[v^k].
\end{align}
By Assumption \ref{hyp: L Lx Lv}, $L_v(X,v) \in \mathds{H}_{\mathds{F}}$, hence
\begin{equation}\label{eq: aux LowerSC of I 4}
	\langle L_v(X,v),v^k-v\rangle_{\mathds{H}_{\mathds{F}}}\to 0.
\end{equation}
By Assumption \ref{hyp: Psi growth}, $\Psi'(X_T)\in \mathds{H}_{\mathds{F}}$, and using the representation $X_T^k - X_T = \int_0^T v^k_t - v_t~ dt$, we obtain 
\begin{equation}\label{eq: aux LowerSC of I 5}
	\Ee[\Psi'(X_T) (X^k_T - X_T)] \to 0.
\end{equation}
By Assumption \ref{hyp: L Lx Lv}, the Cauchy inequality, and the triangle inequality
\begin{align*}
	\left| \Ee\left[ \int_0^T L(X^k_t,v^k_t)-L(X_t,v^k_t)~ dt \right] \right| & \leq C \langle |X^k-X| , 1+|v^k|\rangle_{\mathds{H}_{\mathds{F}}}-
	\\ & \leq C \|X^k-X\|_{\mathds{H}_{\mathds{F}}}\left(T+\|v^k\|_{\mathds{H}_{\mathds{F}}}\right)\to 0.
\end{align*}
Using the previous inequality, \eqref{eq: aux LowerSC of I 4}, \eqref{eq: aux LowerSC of I 5}, and the assumption on $\varpi$, taking $\liminf$ in \eqref{eq: aux LowerSC of I 3}, we obtain 
\[
	I[v] \leq \liminf_{k \in \Nn} I[v^k].
\]
Therefore, $I[\cdot]$ is lower semi-continuous.
\end{proof}

\begin{proposition} \label{proposition: existence single agent}
Suppose that Assumptions \ref{hyp: L jconvex}-
%, \ref{hyp: Psi convex}, \ref{hyp: Psi growth}, \ref{hyp: L Lx Lv}, and 
\ref{hyp: Coercivity} hold.
Given an initial condition $x_0\in\Rr$ and a price process $\varpi \in \mathds{H}_{\mathds{F}}$, there exists an optimal control $v^* \in \mathds{H}_{\mathds{F}}$ that solves \eqref{eq: Single agent min problem}.
Furthermore, under Assumption \ref{hyp: L uniformly convex}, $v^*$ is unique.
\end{proposition}
\begin{proof}
To prove existence, we use the direct method in the calculus of variations.
By Assumption \ref{hyp: Coercivity}, we have
\begin{equation}\label{eq: aux eq Existence Single Agent}
	\alpha\left(v+\frac{\varpi}{2\alpha}\right)^2 - \frac{\varpi^2}{4\alpha}-\beta\leq L(x,v)+v\varpi.
\end{equation}
Since $\varpi \in \mathds{H}_{\mathds{F}}$, select $a$ and $b$ such that
\[
	0<a<\alpha, \quad \tfrac{1}{2(\alpha-a)}\|\varpi\|_{\mathds{H}_{\mathds{F}}}^2 \leq b.
\]
Then, for any $v \in \mathds{H}_{\mathds{F}}$, we have
\[
	0\leq (\alpha - a ) \Ee\left[\int_0^T \left(v_t+\tfrac{1}{2(\alpha-a)}\varpi_t\right)^2 dt\right] + b - \tfrac{1}{2(\alpha-a)}\|\varpi\|_{\mathds{H}_{\mathds{F}}}^2.
\]
The previous inequality, \eqref{eq: aux eq Existence Single Agent}, and $0\leq \Psi$ in Assumption \ref{hyp: Psi convex}, imply
\[
	a \|v\|_{\mathds{H}_{\mathds{F}}}^2 - b - \beta T \leq \Ee\left[ \int_0^T \alpha (v_t)^2 +\varpi_t v_t - \beta ~ dt \right]\leq I[v]
\]	
for all $v \in \mathds{H}_{\mathds{F}}$.
Therefore, $v \mapsto I[v]$ is coercive, and in particular, the infimum in \eqref{eq: Single agent min problem} is finite.
Let $(v^k)_{k\in\Nn}$ in $\mathds{H}_{\mathds{F}}$ be a minimizing sequence; that is,
\[
	\lim_{k\to+\infty} I[v^k]=\inf_{v\in \mathds{H}_{\mathds{F}}} I[v].
\]
By the coercivity of $I[\cdot]$, $(v^k)_{k\in\Nn}$ is bounded in $\mathds{H}_{\mathds{F}}$.
Recall that $\mathds{H}_{\mathds{F}}$ is a Hilbert space, so it is reflexive and, therefore, weakly precompact (\cite{E6}, Appendix D, Theorem 3).
Hence, there exists a subsequence, still denoted by $v^k$, that weakly converges to $v^*\in \mathds{H}_{\mathds{F}}$; that is, for all $w \in \mathds{H}_{\mathds{F}}$
\[
	\langle v^k , w \rangle_{\mathds{H}_{\mathds{F}}} \to \langle v^* , w \rangle_{\mathds{H}_{\mathds{F}}}.
\]
By Proposition \ref{prop: wlsc of I} 
\[
	I[v^*] \leq \liminf_{k\to+\infty} I[v_k] = \lim_{k\to+\infty} I[v_k] =\inf_{v\in \mathds{H}_{\mathds{F}}} I[v].
\]
Therefore, $v^*$ is a minimizer.

To prove uniqueness, denote by $X^*$ the solution of \eqref{eq: agent dynamics with initial condition} with the control variable $v^*$.
Assume that $\tilde{v} \in  \mathds{H}_{\mathds{F}}$ is a minimizer of \eqref{eq: Single agent min problem}, with trajectory $	\tilde{X}$ solving \eqref{eq: agent dynamics with initial condition} for $\tilde{v}$.
Set $Y=\tfrac{1}{2}(X^*+\tilde{X}) $, so that $Y$ satisfies \eqref{eq: agent dynamics with initial condition} for the control $\tfrac{1}{2}(v^*+\tilde{v})$.
Then, by Assumptions \ref{hyp: L jconvex} and \ref{hyp: L uniformly convex}, 
\[
	I\left[\tfrac{1}{2}(v^*+\tilde{v})\right] \leq \tfrac{1}{2} \left(I\left[v^*\right]+I\left[\tilde{v}\right]\right)-\tfrac{\theta}{4} \|v^* - \tilde{v}\|_{\mathds{H}_{\mathds{F}}}^2.
\]
It follows that $\tilde{v}=v^*$ in $\mathds{H}_{\mathds{F}}$, which implies that $\tilde{X}=X^*$.
\end{proof}
The following result provides a characterization of minimizers of $I[\cdot]$.
This condition is a weak form of the Euler-Lagrange equation.

\begin{proposition}\label{Prop: Weak EL Single agent}
Suppose that Assumptions \ref{hyp: Psi growth} and \ref{hyp: L Lx Lv}  hold.
Let $v^*\in \mathds{H}_{\mathds{F}}$ solve \eqref{eq: Single agent min problem}, with the corresponding trajectory $X^*$ solving \eqref{eq: agent dynamics with initial condition}.
Then $(X^*,v^*)$ satisfies
\begin{equation}\label{eq: EL Single agent}
\Ee\left[ \int_0^T \Big(L_x(X^*_t,v^*_t)\delta X_t  + \left(L_v(X^*_t,v^*_t) + \varpi_t\right) \delta v_t \Big)dt + \Psi'(X_T^*) \delta X_T\right]=0
\end{equation}
for all $\delta v \in \mathds{H}_{\mathds{F}}$, where 
\begin{equation}\label{eq: deltaX}
\delta X_t = \int_0^t \delta v_s ds.
\end{equation}
\end{proposition}
\begin{proof}
Let $\epsilon \in\Rr $ and $\delta v\in \mathds{H}_{\mathds{F}}$.
Consider the control $v^*+\epsilon \delta v$ in \eqref{eq: agent dynamics with initial condition}.
The corresponding trajectory is $X^\epsilon_t = X^*_t + \epsilon \delta X_t$.
Because $v^*$ is a minimizer of $I[\cdot]$, the function
\[
	\epsilon \mapsto \Ee\left[\int_0^T\Big( L(X^\epsilon_t ,v^*_t+\epsilon \delta v_t) + \varpi_t (v^*_t+\epsilon \delta v_t)\Big) dt + \Psi(X^\epsilon_T) \right]
\]
has a minimum at $\epsilon=0$; that is,
\begin{align}\label{eq: D=0 single Agent}
	\left.\frac{d}{d\epsilon}\Ee\left[\int_0^T \Big(L(X^\epsilon_t ,v^*_t+\epsilon \delta v_t) + \varpi_t (v^*_t+\epsilon \delta v_t) \Big)dt + \Psi(X^\epsilon_T) \right]\right\rvert_{\epsilon=0}=0.
\end{align}
By Assumption \ref{hyp: L Lx Lv}, the partial derivatives of $L$ evaluated at $(X^\epsilon_t ,v^*_t+\epsilon \delta v_t)$ are integrable w.r.t.
$\Ee[\int_0^T \cdot ~dt]$.
From Assumption \ref{hyp: L Lx Lv} and Young's inequality, we have that
\[
	L(X^\epsilon_t ,v^*_t+\epsilon \delta v_t) + \varpi_t (v^*_t+\epsilon \delta v_t) \leq \tilde{\beta}+(\tilde{\beta}+\frac{1}{2})|v^*_t+\epsilon \delta v_t|^2 + \frac{1}{2}|\varpi_t|^2.
\]
In the same way, Assumption \ref{hyp: Psi growth} guarantees analogous conditions for $\Psi$ at $X^\varepsilon_T$.
Hence, we can differentiate under the integral sign in \eqref{eq: D=0 single Agent} (\cite{Billi}, Theorem 16.8), from which the result follows.
\end{proof}
The formulation presented in Proposition \ref{Prop: Weak EL Single agent} corresponds to the classical second-order characterization of minimizers given by the Euler-Lagrange equations.
As in Hamiltonian mechanics, this second-order characterization has an equivalent first-order formulation.
For this first-order characterization, we use the adjoint equation (see \eqref{eq: backward SDE for P}).

\begin{proposition}\label{proposition: Hamilton System}
Suppose $L\in C^1(\Rr^2;\Rr)$ and Assumptions \ref{hyp: Psi growth} and \ref{hyp: L uniformly convex} hold.
%Suppose that $H_p$ and $H_x$ are uniformly Lipschitz in $(x,p)$.
Given $x_0\in\Rr$, assume that $(v^*,X^*)$ solves \eqref{eq: EL Single agent}, where $v^*\in \mathds{H}_{\mathds{F}}$, and $X^*$ solves \eqref{eq: agent dynamics with initial condition} for $v^*$.
Then, the backward SDE
\begin{equation}\label{eq: backward SDE for P}
\begin{cases}
dP_t = - L_x(X^*_t,v_t^*) dt + Z_t dW_t
\\
P_T=\Psi'(X^*_T)
\end{cases}
\end{equation}
has a unique solution $(P,Z)$ on $[0,T]$, where $P,~ Z\in \mathds{H}_{\mathds{F}}$.
Furthermore, 
\begin{equation}\label{eq: Implicit control from H system}
	P = - L_v(X^*,v^*)-\varpi,
\end{equation}
and $(X^*,P,Z)$  solves, on $[0,T]$, the forward-backward SDE system
\begin{equation}\label{eq: Hamiltonian System}
	\begin{cases} 
	dX_t = -H_p(X_t,P_t+\varpi_t) dt
	\\
	X_0=x_0
	\\
	dP_t = H_x(X_t,P_t+\varpi_t)dt  + Z_t dW_t
	\\
	P_T=\Psi'(X_T).
	\end{cases}
\end{equation}
\end{proposition}
\begin{proof}
Assumption \ref{hyp: Psi growth} implies that $\Psi'(X^*_T)\in \mathds{L}^2_T(\Rr)$, and the continuity of $L_x$ guarantees the adaptability of $L_x(X^*_t,v_t^*)$ w.r.t. $\mathds{F}$. Notice that this term is independent of $P$ and $Z$.
Hence, Theorem 2.1 in \cite{BackwardSDE1997} guarantees the existence and uniqueness of $(P,Z)$ solving  \eqref{eq: backward SDE for P}.

Let $\delta v\in \mathds{H}_{\mathds{F}}$ and $\delta X$ according to \eqref{eq: deltaX}.
Then, because $\delta X_0 =0$, using \eqref{eq: backward SDE for P}, we have
\begin{align*}
&\Ee[\Psi'(X_T^*) \delta X_T] =\Ee[P_T\delta X_T] =\Ee\left[ \int_0^T d\left(P_t \delta X_t\right) \right]
\\
&=\Ee\left[ \int_0^T dP_t \delta X_t + P_t \delta v_t dt \right] =\Ee\left[ \int_0^T  -L_x(X^*_t,v_t^*)\delta X_t dt + P_t \delta v_t dt  + Z_t\delta X_t   dW_t \right].
\end{align*}
From the previous identity and \eqref{eq: EL Single agent}, we get
\begin{align*}
&\Ee[\Psi'(X_T^*) \delta X_T] =\Ee\left[ \int_0^T  (L_v(X_t^*,v_t^*)+\varpi_t + P_t)\delta v_t dt + Z_t\delta X_t   dW_t + \Psi'(X_T^*)\delta X_T \right].
\end{align*}
Recall that $Z,\delta X \in \mathds{H}_{\mathds{F}}$, which implies that $\Ee\left[ \int_0^T  Z_t\delta X_t   dW_t \right]=0$.
Hence, we conclude that, for all $\delta v \in \mathds{H}_{\mathds{F}}$,
\[
	\Ee\left[ \int_0^T  (L_v(X_t^*,v_t^*)+\varpi_t + P_t)\delta v_t dt\right]=\langle L_v(X^*,v^*)+\varpi + P , \delta v \rangle_{\mathds{H}_{\mathds{F}}}= 0.
\]	
Therefore, $P_t = - L_v(X_t^*,v_t^*)-\varpi_t$, from which Assumption \ref{hyp: L uniformly convex} and \eqref{eq: Aux Legendre T  v p} imply that $(X^*,P,Z)$ solves \eqref{eq: Hamiltonian System}.
\end{proof}
\begin{remark}
Notice that \eqref{eq: Hamiltonian System} is independent of the optimal control $v^*$.
Hence, if \eqref{eq: Hamiltonian System} has a unique solution and \eqref{eq: Implicit control from H system} is invertible, we obtain explicit expressions for the optimal control.
This is the case, for instance, when $L$ and $\Psi$ are quadratic, as we illustrate in Section \ref{Sec: LQ model}.
%The conclusion of Proposition \ref{proposition: Hamilton System} gives an implicit formula for the optimal control $v$.
% If $v\mapsto L_v(x,v)$ is invertible, for instance, when $L$ is quadratic on $v$, this relation is explicit.
% Hence, the solution $P$ of \eqref{eq: Hamiltonian System} provides the optimal control in the feedback form $v=v(P,X,\varpi)$.
\end{remark}
Next, we give conditions for the converse of Proposition \ref{proposition: Hamilton System} to hold.
\begin{proposition}\label{proposition: Hamilton System to EL}
Assume that $L\in C^2(\Rr^2;\Rr)$ is strictly convex in $v$, $\Psi \in C^1(\Rr)$, and Assumptions \ref{hyp: Coercivity}, and \ref{hyp: H_p growth} hold.
Let $(X^*,P,Z)$ solve \eqref{eq: Hamiltonian System}, where $X^*,~P,~Z\in  \mathds{H}_{\mathds{F}}$.
Then, $v^*:=-H_p(X^*,P+\varpi)$ and $X^*$ satisfy   \eqref{eq: EL Single agent}.
Furthermore,  $P = - L_v(X^*,v^*)-\varpi$.
\end{proposition}
\begin{proof}
From Assumption \ref{hyp: H_p growth}, we have $v^* \in \mathds{H}_{\mathds{F}}$. The first equation in \eqref{eq: Hamiltonian System} states that $X^*$ solves \eqref{eq: agent dynamics with initial condition} for the control $v^*$.
Then, by the strict convexity of $L$ in $v$ and Assumption \ref{hyp: Coercivity}, \eqref{eq: Aux Legendre T  v p} gives that $P = - L_v(X^*,v^*)-\varpi$ and $L_x(X^*,v^*) = -H_x(X^*,P+\varpi)$.
Take $\delta v \in \mathds{H}_{\mathds{F}}$ and $\delta X$, as in \eqref{eq: deltaX}, and multiply the previous identities to obtain
\begin{equation}\label{eq: aux eq 1}
	L_x(X^*,v^*)\delta X = -H_x(X^*,P+\varpi) \delta X, \quad \mbox{and} \quad  (L_v(X^*,v^*)+\varpi) \delta v = -P \delta v.
\end{equation}
Integrating on $[0,T]$ the relation $d\Big( P_t \delta X_t \Big) = dP_t \delta X_t + P_t \delta v_t dt$, recalling that $\delta X_0=0$, and replacing \eqref{eq: aux eq 1}, we have
\[
	P_T \delta X_T = \int_0^T dP_t \delta X_t -(L_v(X^*_t,v^*_t)+\varpi_t)\delta v_t dt.
\]
Using the third equation in \eqref{eq: Hamiltonian System}, the previous expression becomes
\[
	P_T \delta X_T = \int_0^T H_x(X^*_t,P_t+\varpi_t) \delta X_tdt +Z_t\delta X_t dW_t -(L_v(X^*_t,v^*_t)+\varpi_t)\delta v_t dt.
\]
Replacing the terminal condition for $P$ in \eqref{eq: Hamiltonian System}, using \eqref{eq: Aux Legendre T  v p}, taking expectation and recalling that $\Ee\left[ \int_0^T  Z_t\delta X_t   dW_t \right]=0$, we obtain
\begin{align*}
&\Ee\left[ \int_0^T L_x(X_t^*,v_t^*)\delta X_t + (L_v(X_t^*,v_t^*)+\varpi_t) \delta v_t ~ dt + \Psi'(X_T^*)\delta X_T \right] =0.
\end{align*}
Because $\delta v \in \mathds{H}_{\mathds{F}}$ is arbitrary, $(X^*,v^*)$ solves \eqref{eq: EL Single agent}.
\end{proof}
Notice that Proposition \ref{proposition: Hamilton System} guarantees the existence of solutions to the system \eqref{eq: backward SDE for P} and \eqref{eq: Hamiltonian System}, but it only states the uniqueness of solutions to the system \eqref{eq: backward SDE for P}.
The following proposition states a uniqueness result for \eqref{eq: Hamiltonian System}.
\begin{proposition}
Suppose that Assumptions \ref{hyp: L jconvex}, \ref{hyp: Psi convex}, \ref{hyp: Psi growth}, \ref{hyp: Coercivity} and \ref{hyp: H_p growth} hold.
Assume that $L$ is strictly convex in $v$.
Let $H$ be given by \eqref{def: Legendre trans}.
Suppose $H\in C^1(\Rr^2;\Rr)$,
%is concave in $x$
 and either 
\[
	\begin{array}{lcl}
	L \mbox{ is strictly convex in } (x,v) & \mbox{ or } & ~\Psi \mbox{ is strictly convex, and } H_p,~ H_x\\
	 &  &   \mbox{ are uniformly Lipschitz in } (x,p).
	\end{array}
\]
%{\color{red} $H$ is strictly convex in $p$, with $H_p$ Lipschitz in $x$, 
%
%or 
%
%$\Psi$ is strictly convex, $H_p$ and $H_x$ Lipschitz in $(x,p)$.}
Then, the solution to \eqref{eq: Hamiltonian System} is unique.
\end{proposition}
\begin{proof}
Let $(X,P,Z)$ and $(\tX,\tP,\tZ)$ solve \eqref{eq: Hamiltonian System}.
Let 
\begin{align*}
v=-H_p(X,P+\varpi), \quad 
H_x=H_x(X_t,P_t+\varpi_t), \quad
L_x=L_x(X,v), \quad 
L_v=L_v(X,v).
\\
\tilde{v}=-H_p(\tX,\tP+\varpi), \quad 
\tH_x=H_x(\tX_t,\tP_t+\varpi_t), \quad
\tilde{L}_x=L_x(\tilde{X},\tilde{v}), \quad
\tilde{K}_v=L_v(\tilde{X},\tilde{v}).
\end{align*}
By Assumption \ref{hyp: H_p growth}, $v,\tilde{v} \in \mathds{H}_{\mathds{F}}$.
Because of the strict convexity of $L$ in $v$ and Assumption \ref{hyp: Coercivity}, using \eqref{eq: Aux Legendre T  v p}, we have
\begin{equation}\label{eq: aux Uniqueness H system 0}
	L_v = -(P+\varpi), ~ \tilde{L}_v=-(\tP+\varpi), ~ H_x=-L_x, ~ \tilde{H}_x = -\tilde{L}_x.
\end{equation}
By Assumption \ref{hyp: Psi growth}, $\Psi'(X_T),\Psi'(\tilde{X}_T) \in \mathds{L}^2_T(\Rr)$ (Theorem 2.1, \cite{BackwardSDE1997}). Hence, using \eqref{eq: Hamiltonian System} and It\^o's product rule, we get
\begin{align}\label{eq: aux Uniqueness H system 1}
&\Ee \left[\Big(\Psi'(X_T)-\Psi'(\tX_T)\Big)\Big(X_T-\tX_T\Big)\right]
\\
&=\Ee \left[\int_0^T d\left( (P_t-\tP_t)(X_t-\tX_t)\right)\right] \nonumber
\\
&=\Ee \left[\int_0^T (H_x - \tH_x)(X_t-\tX_t)dt + (Z_t-\tZ_t)(X_t-\tX_t)dW_t - (P_t-\tP_t)(\tilde{v}_t-v_t)dt\right].\nonumber
\end{align}
Recalling that $\Ee \left[\int_0^T (Z_t-\tZ_t)(X_t-\tX_t)dW_t \right]=0$ and using \eqref{eq: aux Uniqueness H system 0}, we obtain
\begin{align*}
&\Ee \left[\int_0^T (H_x - \tH_x)(X_t-\tX_t)dt + (Z_t-\tZ_t)(X_t-\tX_t)dW_t - (P_t-\tP_t)(\tilde{v}_t-v_t)dt\right]
\\
&=\Ee \left[\int_0^T (\tilde{L}_x-L_x)(X_t-\tX_t)dt  - (\tilde{L}_v-L_v)(\tilde{v}_t-v_t)dt\right],
\end{align*}
and by Assumption \ref{hyp: L jconvex} (\cite{bauschke2017convex}, Proposition 17.7)
\begin{equation}\label{eq: aux Uniqueness H system 2}
	\Ee \left[\int_0^T (\tilde{L}_x-L_x)(X_t-\tX_t)dt  - (\tilde{L}_v-L_v)(\tilde{v}_t-v_t)dt\right]\leq 0.
\end{equation}
In the same way, the convexity of $\Psi$ (see Assumption \ref{hyp: Psi convex}) implies
\begin{align}\label{eq: aux Uniqueness H system 3}
0 \leq \Ee \left[\Big(\Psi'(X_T)-\Psi'(\tX_T)\Big)\Big(X_T-\tX_T\Big)\right].
\end{align}
Hence, from \eqref{eq: aux Uniqueness H system 1}, \eqref{eq: aux Uniqueness H system 2} and \eqref{eq: aux Uniqueness H system 3}, it follows that 
\begin{align}\label{eq:Aux Uniqueness Hamiltonian System sol}
	0 &\leq \Ee \left[\Big(\Psi'(X_T)-\Psi'(\tX_T)\Big)\Big(X_T-\tX_T\Big)\right] 
	\\
	&\leq \Ee \left[\int_0^T (\tilde{L}_x-L_x)(X_t-\tX_t)dt  - (\tilde{L}_v-L_v)(\tilde{v}_t-v_t)dt\right]\leq 0.\nonumber
\end{align}
Now we use a characterization of strict convexity provided in \cite{bauschke2017convex}, Proposition 17.10:  If $L$ is strictly convex in $(x,v)$, \eqref{eq:Aux Uniqueness Hamiltonian System sol} implies $X=\tX$ and $v=\tilde{v}$, from which \eqref{eq: aux Uniqueness H system 0} implies $P=\tP$ and $Z=\tZ$ follows.
%By the uniqueness result in Theorem \ref{Thm: PardouxPeng} for the BSDE 
%\[
%	\begin{cases} 
%	dP_t = H_x(X_t,P_t+\varpi_t)dt  + Z_t dW_t
%	\\
%	P_T=\Psi'(X_T) 
%	\end{cases}
%\]
%on $[0,T]$, we conclude that $Z=\tZ$.
On the other hand, if $\Psi$ is strictly convex, \eqref{eq:Aux Uniqueness Hamiltonian System sol} implies $\tX_T=X_T$.
Therefore, both $(X,P,Z)$ and $(\tX,\tP,\tZ)$ solve the BSDE
\begin{equation*}
	\begin{cases} 
	dX_t = -H_p(X_t,P_t+\varpi_t) dt
	\\
	X_T=\tX_T,
	\\
	dP_t = H_x(X_t,P_t+\varpi_t) dt + Z_t dW_t
	\\
	P_T=\Psi'(\tX_T)
	\end{cases}
\end{equation*}
for all $t\in[0,T]$.
The Lipschitz condition in both $H_p$ and $H_x$ allows us to use Theorem 2.1 in \cite{BackwardSDE1997} to conclude that $(X,P,Z)=(\tX,\tP,\tZ)$.
\end{proof}
Propositions \ref{Prop: Weak EL Single agent} and \ref{proposition: Hamilton System} show that the existence of solutions to \eqref{eq: Hamiltonian System} is a necessary condition for the existence of solutions to \eqref{eq: Single agent min problem}.
In the next result, we consider conditions for \eqref{eq: Hamiltonian System} to be sufficient.

\begin{proposition}%\label{proposition: Hamiltonian System solution is minimizer}
Suppose that Assumptions \ref{hyp: L jconvex}, \ref{hyp: Psi convex}, \ref{hyp: Coercivity}, and \ref{hyp: H_p growth} hold. Assume further that $L$ is strictly convex in $v$. Let $(X^*,P,Z)$ solve \eqref{eq: Hamiltonian System} and define $v^*=-H_p(X^*,P+\varpi)$.
Then $v^*$ solves \eqref{eq: Single agent min problem}.
\end{proposition}
\begin{proof}
By Assumption \ref{hyp: H_p growth}, $v^* \in \mathds{H}_{\mathds{F}}$.
Let $v \in \mathds{H}_{\mathds{F}}$ and $X$ solve \eqref{eq: agent dynamics with initial condition} for $v$.
From Proposition \ref{proposition: Hamilton System to EL}, we have $L_v(X^*,v^*) = - P - \varpi$.
By the convexity of $L$ in $(x,v)$, we have 
\[
		L(X^*_t,v^*_t) + L_x(X^*_t,v^*_t)(X_t-X^*_t)+L_v(X^*_t,v^*_t)(v_t-v^*_t) \leq L(X_t,v_t)
\]
By Assumption \ref{hyp: Coercivity} and the strict convexity of $L$ in $v$, from \eqref{eq: Aux Legendre T  v p}, we get
\begin{align*}
	&L_x(X^*_t,v^*_t)(X_t-X^*_t)+L_v(X^*_t,v^*_t)(v_t-v^*_t)
	\\
	&=-H_x(X^*_t,P_t+\varpi_t) (X_t-X^*_t)-(P_t+\varpi_t)(v_t-v^*_t).
\end{align*}
Hence,
\begin{equation}\label{eq: Aux Hamiltonian is Sufficient}
	L(X^*_t,v^*_t)+\varpi v_t^*	-H_x(X^*_t,P_t+\varpi_t)(X_t-X^*_t)-(P_t)(v_t-v^*_t)\leq L(X_t,v_t)+\varpi_t v_t.
\end{equation}
Using \eqref{eq: agent dynamics with initial condition} and \eqref{eq: Hamiltonian System}, we compute
\begin{align*}
	d\left(P_t (X_t-X^*_t)\right)&=dP_t(X_t-X^*_t) + P_t (dX_t-dX^*_t)
	\\
	&=H_x(X^*_t,P_t+\varpi_t)(X_t-X^*_t)dt  + Z_t(X_t-X^*_t) dW_t+P_t (v_t-v^*_t).
\end{align*}
Taking $\Ee[\int_0^T \cdot ~ dt]$ in the previous identity, recalling that $
\Ee\left[\int_0^T Z_t(X_t-X^*_t) dW_t\right]=0$, and using the terminal condition for $P$ in \eqref{eq: Hamiltonian System} and the initial condition for $X$ and $X^*$ in \eqref{eq: agent dynamics with initial condition}, we get
\begin{align*}
	\Ee\left[ P_T (X_T-X^*_T)\right]=\Ee\left[ \int_0^T H_x(X^*_t,P_t+\varpi_t)(X_t-X^*_t)dt+P_t (v_t-v^*_t)dt\right].
\end{align*}
From the previous identity, \eqref{eq: Aux Hamiltonian is Sufficient}, and the convexity of $\Psi$ (see Assumption \ref{hyp: Psi convex}), we conclude that 
\[
	I[v^*]\leq I[v]
\]
for arbitrary $v$.
The result follows.
\end{proof}
%----------------------------------------------------------------------------------------------------------------------------------

\section{The N-agent problem}\label{sec: NAgent problem}
Here, we introduce a minimization problem that aggregates the costs of all agents considered in the previous section and is constrained by the total supply.
By aggregating the costs of all agents, we obtain an equivalent variational problem independent of the price.
We show the existence of minimizers and obtain the price as the Lagrange multiplier for the supply constraint.

Let $\bx_0 \in \Rr^N$ be the initial configuration of $N$ agents.
Given the controls $\bv \in \mathds{H}_{\mathds{F}}^N$, consider the dynamics for $N$ agents
\begin{equation}\label{eq: N agents dynamics with initial condition}
\begin{cases} d\bX_t=\bv_t dt ,~ t\in[0,T]\\ \bX_0=\bx_0,\end{cases}
\end{equation}
where $\bX=(X^1,\ldots,X^N)$.
If the price is known, the functional of a representative agent in the minimization problem \eqref{eq: Single agent min problem} depends on the actions of other agents through the price.
To solve \eqref{eq: Single agent min problem}, each agent looks for its optimal control $v^*$, and this control is coupled with the control of other agents through the balance condition \eqref{eq:Balance condition}.
Hence, as long as the balance condition is satisfied, the vector $\bv^*:=({v^*}^1,\ldots,{v^*}^i)$, consisting of the optimal controls for each agent, is an optimal control for the following minimization problem
\begin{gather*}
\inf_{\bv\in \mathds{H}_{\mathds{F}}^N} \frac{1}{N} \sum_{i=1}^N \Ee\left[\int_0^T L(X_t^i,v^i_t) + \varpi_t v_t^i~ dt + \Psi(X^i_T) \right]
\\
\mbox{subject to }  \quad \bX ~ \mbox{ solves } ~\eqref{eq: N agents dynamics with initial condition}.
\nonumber
\end{gather*}
Reciprocally, as long as the balance condition is satisfied, any optimal control $\bv^*$ of the previous minimization problem provides, through its components ${v^*}^i$, for $1\leq i \leq N$, an optimal control for \eqref{eq: Single agent min problem}.
Therefore, Problem \ref{problem: problem} is equivalent to the following
\begin{gather}\label{eq: Nagent min problem}
\inf_{\bv\in \mathds{H}_{\mathds{F}}^N} \frac{1}{N} \sum_{i=1}^N \Ee\left[\int_0^T L(X_t^i,v^i_t) + \varpi_t v_t^i~ dt + \Psi(X^i_T) \right]
\\
\mbox{subject to }  \quad \frac{1}{N}\sum\limits_{i=1}^N v^i=Q, ~\mbox{ and } ~ \bX ~ \mbox{ solves } ~\eqref{eq: N agents dynamics with initial condition}.
\nonumber
\end{gather}
Substituting the balance condition into the expression to minimize in \eqref{eq: Nagent min problem}, we get 
\begin{equation}\label{eq: Aux N Agent reduced}
\frac{1}{N} \sum_{i=1}^N \Ee\left[\int_0^T L(X_t^i,v^i_t) + \varpi_t Q_t ~ dt + \Psi(X^i_T) \right].
\end{equation}
Let
\begin{equation*}
%\label{def: N agent functional}
	I_N[\bv]:=\frac{1}{N} \sum_{i=1}^N \Ee\left[\int_0^T L(X_t^i,v^i_t) dt + \Psi(X^i_T) \right].
\end{equation*}
Since the expression $\langle \varpi, Q\rangle_{\mathds{H}_{\mathds{F}}}$ in \eqref{eq: Aux N Agent reduced} is independent of $\bv$, we can drop this term and obtain that \eqref{eq: Nagent min problem} is equivalent to the following problem

\begin{problem}\label{problem:problem no price}
%Let $N$ be the number of agents.
% Let the supply, $Q$, be a stochastic process adapted to $\mathds{F}$.
% Let $L \in C^1(\Rr^2;\Rr)$ be a non-negative Lagrangian, and $\Psi \in C^1(\Rr)$ be a non-negative terminal cost.
%Assume that at time $t=0$, each agent $i$ owns a quantity $x_0^i\in\Rr$ of the commodity.

Find a vector of control processes $\bv^*\in\mathds{H}_{\mathds{F}}^N$ that attains the following
\begin{gather}\label{eq: reduced Nagent min problem}
\inf_{\bv\in \mathds{H}_{\mathds{F}}^N}~ I_N[\bv]
\\
\mbox{subject to }  \quad \frac{1}{N}\sum\limits_{i=1}^N v^i=Q, ~ \mbox{ and } ~ \bX ~ \mbox{ solves } ~\eqref{eq: N agents dynamics with initial condition}.\nonumber
\end{gather}

\end{problem}

%\begin{gather}\label{eq: reduced Nagent min problem}
%I_N[\bv^*]=\inf_{\bv\in \mathds{H}_{\mathds{F}}^N}~ I_N[\bv]
%\\
%\mbox{subject to }  \quad \frac{1}{N}\sum\limits_{i=1}^N v^i=Q, ~ \mbox{ and } ~ \bX ~ \mbox{ solves } ~\eqref{eq: N agents dynamics with initial condition}.\nonumber
%\end{gather}
The next proposition shows that this problem has a solution; that is, there exists $\bv^*\in \mathds{H}_{\mathds{F}}^N$ such that $I_N[\bv^*]$ attains the infimum in \eqref{eq: reduced Nagent min problem} and satisfies the constraints.

\begin{proposition}\label{Prop: Existence N constrained}
Let $Q\in \mathds{H}_{\mathds{F}}$.
Suppose that Assumptions \ref{hyp: L jconvex}, \ref{hyp: Psi convex}, \ref{hyp: Psi growth}, \ref{hyp: L Lx Lv}, and \ref{hyp: Coercivity} hold.
Given an initial condition $\bx_0 \in \Rr^N$, there exists an optimal control $\bv^* \in \mathds{H}_{\mathds{F}}^N$ that solves Problem \ref{problem:problem no price}.
Furthermore, under Assumption \ref{hyp: L uniformly convex}, $\bv^* $ is unique.
\end{proposition}
\begin{proof}
We follow the direct method in the calculus of variations to prove existence.
Define the set of admissible controls
\begin{equation*}%\label{def: Admissible trajectories}
	\mathcal{C}=\left\{ \bv \in \mathds{H}_{\mathds{F}}^N:~  \frac{1}{N}\sum_{i=1}^N v^i=Q,\right\}.
\end{equation*}
Notice that $\mathcal{C}$ is a convex set.
Also, this set is not empty because $v^i=Q$ for $1\leq i \leq N$ is an element of $\mathcal{C}$.
The set $\mathcal{C}$ is also closed because any sequence $(\bv^k)_{k\in\Nn}$ in $\mathcal{C}$ that converges to $\bv$ in $\mathds{H}_{\mathds{F}}^N$ satisfies
\begin{align*}
	\left\Vert \frac{1}{N} \sum_{i=1}^N v^i - Q \right\Vert^2_{\mathds{H}_{\mathds{F}}} \leq \frac{2}{N} \sum_{i=1}^N \left\Vert v^i - v^{i,k} \right\Vert^2_{\mathds{H}_{\mathds{F}}} \to 0.
\end{align*}
By a similar argument to that used in the proof of Proposition \ref{proposition: existence single agent}, Assumption \ref{hyp: Coercivity} implies that
\[
	\frac{\alpha}{N} \|\bv\|^2_{\mathds{H}_{\mathds{F}}^N} - \beta T \leq I_N[{\bv}],
\]
for all $\bv\in \mathds{H}_{\mathds{F}}^N$.
Therefore, $\bv \mapsto I_N[\bv]$ is coercive and bounded from below.
In particular, the infimum in \eqref{eq: reduced Nagent min problem} is finite.
Let $(\bv^k)_{k\in\Nn}$ in $\mathcal{C}$ be a minimizing sequence of \eqref{eq: reduced Nagent min problem}; that is,
\[
	\lim_{k\to+\infty} I_N[\bv^k]=\inf_{\bv\in \mathcal{C}} I_N[\bv].
\]
By coercivity of $I_N[\cdot]$, $(\bv^k)_{k\in\Nn}$ is bounded in $\mathds{H}_{\mathds{F}}^N$.
Because $\mathds{H}_{\mathds{F}}$ is a Hilbert space, $\mathds{H}_{\mathds{F}}^N$ is also a Hilbert space.
Hence, let $\bv^*\in \mathds{H}_{\mathds{F}}^N$ be a control for which there is a subsequence, still denoted by $\bv^k$, that weakly converges to $\bv^*$.
Since $\mathcal{C}$ is convex and closed, by Mazur's theorem (\cite{kurdila2006convex}, Theorem 7.2.4), it is weakly closed.
Therefore, $\bv^*\in \mathcal{C}$.
Arguing as in Proposition \ref{prop: wlsc of I} using Assumptions \ref{hyp: L jconvex}, \ref{hyp: Psi convex}, \ref{hyp: Psi growth}, and \ref{hyp: L Lx Lv}, we have that $I_N$ is weakly lower semi-continuous.
Hence,
\[
	I_N[\bv^*] \leq \liminf_{k\to+\infty} I_N[\bv^k] = \lim_{k\to+\infty} I_N[\bv^k] =\inf_{\bv\in\mathcal{C}} I_N[\bv].
\]
Accordingly, $\bv^*$ is a minimizer.
The uniqueness of $\bv^*$  follows from Assumption \ref{hyp: L uniformly convex} and a similar argument to the one in the proof of Proposition \ref{proposition: existence single agent}.
%To prove it is unique, assume there exists another minimizer $\tilde{\bv} \in  \mathds{H}_{\mathds{F}}^N$.
%Set $\bY=\tfrac{1}{2}(\bX+\tilde{\bX}) $, so $	d\bY_t = \tfrac{1}{2}(\bv_t+\tilde{\bv}_t)dt$ and 
%\[
%	I_N\left[\tfrac{1}{2}(\bv+\tilde{\bv})\right] \leq \tfrac{1}{2} \left(I_N\left[\bv\right]+I_N\left[\tilde{\bv}\right]\right)-\tfrac{\theta}{4} \|\bv - \tilde{\bv}\|_{\mathds{H}_{\mathds{F}}^N}^{2/N}
%\]
%so it follows that $\bv^*=\tilde{\bv}^*$ in $\mathds{H}_{\mathds{F}}^N$, which implies that $\bX^*=\tilde{\bX}^*$.
\end{proof}
The following lemma characterizes the orthogonal complement of the elements in $ \mathds{H}_{\mathds{F}}^N$, whose entries add to zero.
We will use this lemma to prove the existence of a Lagrange multiplier.
\begin{lemma}\label{lemma: Orthogonal to zero mean}
Let  $\mathcal{Z}=\left\{ \bw\in \mathds{H}_{\mathds{F}}^N:~  \sum_{i=1}^N w^i =0 \right\}$.
Denote by $\mathcal{Z}^\perp$ the orthogonal complement of the set $\mathcal{Z}$ with respect to $\langle \cdot , \cdot \rangle_{\mathds{H}_{\mathds{F}}^N} $.
Then, $\mathcal{Z}^\perp=\left\{ \bv\in \mathds{H}_{\mathds{F}}^N:~ \bv=\bar{v} \mathds{1}_N,~ \bar{v} \in \mathds{H}_{\mathds{F}} \right\}$, where $\mathds{1}_N=(1,\ldots,1)\in\Rr^N$.
%
%$\bv \in \mathds{H}_{\mathds{F}}^N$ be such that 
%\[
%	\langle \bv , \delta \bv \rangle_{\mathds{H}_{\mathds{F}}^N} =0 ~ \mbox{ for all } ~\delta \bv\in  \mathds{H}_{\mathds{F}}^N ~\mbox{ satisfying } ~\frac{1}{N}\sum_{i=1}^N \delta v^i =0.
%\]
%Then, there exists a unique $\bar{v} \in \mathds{H}_{\mathds{F}}$ such that $v^i=\bar{v}$ for $1 \leq i \leq N$.
\end{lemma}
\begin{proof}
Let $\bv \in \mathcal{Z}^\perp$.
For $\delta \bw \in \mathds{H}_{\mathds{F}}^N$, define $ \bw = \delta \bw - \left(\frac{1}{N} \sum\limits_{i=1}^N \delta w^i \right)\mathds{1}_N$.
Then, $\bw \in \mathcal{Z}$, which implies that $\langle \bv , \bw \rangle_{\mathds{H}_{\mathds{F}}^N} =0.$ Writing
\begin{align*}
	&\sum_{i=1}^N \bigg\langle \delta w^i  , v^i - \frac{1}{N} \sum_{k=1}^N v^k \bigg\rangle_{\mathds{H}_{\mathds{F}}}=\Ee\left[ \int_0^T \sum_{i=1}^N \delta w_t^i \left( v_t^i - \frac{1}{N} \sum_{k=1}^N v_t^k \right) \right] 
	\\
	&= \Ee\left[ \int_0^T \sum_{i=1}^N v_t^i \left( \delta w_t^i - \frac{1}{N} \sum_{k=1}^N \delta w_t^k \right) \right] =\Ee\left[ \int_0^T \sum_{i=1}^N v_t^i  w_t^i dt \right] =\langle \bv ,  \bw \rangle_{\mathds{H}_{\mathds{F}}^N},
\end{align*}
the orthogonality between $\bv$ and $\bw$ implies that
\begin{align*}
	&\sum_{i=1}^N \bigg\langle \delta w^i  , v^i - \frac{1}{N} \sum_{k=1}^N v^k \bigg\rangle_{\mathds{H}_{\mathds{F}}}=0.
\end{align*}
Because in the previous identity $\delta \bw$ is arbitrary, we conclude that $\bar{v}:=\frac{1}{N} \sum_{k=1}^N v^k \in \mathds{H}_{\mathds{F}}$ satisfies $v^i = \bar{v} $, for  $1\leq i \leq N$; that is, $\bv = \bar{v}\mathds{1}_N$, where $\bar{v} \in \mathds{H}_{\mathds{F}}$.
On the other hand, let $\bv = \bar{v}\mathds{1}_N$, where $\bar{v} \in \mathds{H}_{\mathds{F}}$, and let $\bw \in \mathcal{Z}$.
Then,
\[
	\langle \bv , \bw \rangle_{\mathds{H}_{\mathds{F}}^N} = \sum_{i=1}^N \langle v^i , w^i \rangle_{\mathds{H}_{\mathds{F}}} =\sum_{i=1}^N \Ee \left[ \int_0^T \bar{v}_t w^i_t dt \right] = \Ee\left[\int_0^T \left( \sum_{i=1}^N w^i_t \right) \bar{v}_t dt \right]=0,
\]
which implies that $\bv \in \mathcal{Z}^\perp$.
This completes the proof.
\end{proof}
Next, we prove the existence of a Lagrange multiplier corresponding to the balance condition.
This Lagrange multiplier uniquely defines the price.

\begin{proposition}\label{Proposition: existence of the price}
Suppose that Assumptions \ref{hyp: Psi growth} and \ref{hyp: L Lx Lv} hold.
Let $\bv^* \in \mathds{H}_{\mathds{F}}^N$ solve Problem \ref{problem:problem no price} with the corresponding trajectory $\bX^*$.
For $1\leq i \leq N$, let $P^i ,  Z^i \in \mathds{H}_{\mathds{F}}$ solve, on $[0,T]$,
\begin{equation}\label{eq: N Backward P}
\begin{cases}
dP_t^i = - L_x({X^*_t}^i,{v_t^*}^i) dt + Z^i_t dW_t
\\
P_T^i=\Psi'({X^*_T}^i).
\end{cases}
\end{equation}
Then, there exists a unique $\Pi \in \mathds{H}_{\mathds{F}}$ that satisfies
\begin{equation}\label{eq: N Lagrange Multiplier}
	\Pi=P^i + L_v({X^*}^i,{v^*}^i) \quad \mbox{for }  1 \leq i \leq N.
\end{equation}
Hence, 
\begin{equation}\label{eq:MultiplierFormula}
	\Pi=\frac{1}{N} \sum_{i=1}^N P^i + L_v({X^*}^i,{v^*}^i), ~ \mbox{ and }~ \Pi_T = \frac{1}{N} \sum_{i=1}^N \Psi'({X_T^*}^i) + L_v({X_T^*}^i,{v_T^*}^i).
\end{equation}
\end{proposition}
\begin{proof}
%Let $\delta \bv \in \mathds{H}_{\mathds{F}}^N$ be such that $\sum_{i=1}^N \delta v^i =0$, 
Let $\delta \bv \in \mathcal{Z}$, and define $\delta \bX$ according to \eqref{eq: deltaX}.
Then, for all $\epsilon >0$, according to \eqref{eq: N agents dynamics with initial condition}, the process $\bX^\epsilon=\bX^* + \epsilon \delta \bX$ is driven by $\bv^* + \epsilon \delta \bv$. Notice that $\bv^* + \epsilon \delta \bv \in \mathcal{C}$ because $\bv^*$ satisfies the balance condition, and hence
\[
	\frac{1}{N}\sum\limits_{i=1}^N \left({v^*}^i+\epsilon \delta v^i\right) =Q.
\]
Thus, the function $	\epsilon \mapsto I_N[\bv^* + \epsilon \delta \bv]$ attains a minimum at $\epsilon=0$.
Therefore, 
\[
	\left.\frac{d}{d\epsilon}   I_N[\bv^* + \epsilon \delta \bv] \right\rvert_{\epsilon=0}=0.
\]
Proceeding as in the proof of Proposition \ref{Prop: Weak EL Single agent}, using  Assumptions \ref{hyp: Psi growth} and \ref{hyp: L Lx Lv}, we conclude that
\begin{equation}\label{eq: First Variation reduced N agent min}
	\frac{1}{N} \sum_{i=1}^N \Ee\left[ \int_0^T L_x({X_t^*}^i,{v_t^*}^i) \delta X_t^i + L_v({X_t^*}^i,{v_t^*}^i) \delta v_t^i ~ dt + \Psi'({X_T^*}^i) \delta X_T^i \right] =0.
\end{equation}
Now, for $1\leq i \leq N$, consider the following BSDE on $[0,T]$
\begin{equation}\label{eq:Aux Multiplier existence}
\begin{cases}
dP_t^i = - L_x({X^*_t}^i,{v_t^*}^i) dt + Z^i_t dW_t
\\
P_T^i=\Psi'({X^*_T}^i).
\end{cases}
\end{equation}
Assumption \ref{hyp: Psi growth} guarantees that $\Psi'({X^*_T}^i) \in \mathds{L}^2_T(\Rr)$, so we use Theorem 2.1 in \cite{BackwardSDE1997}, and we denote by $(P^i,Z^i)$ the unique solution of \eqref{eq:Aux Multiplier existence}.
By applying It\^o's product rule to $P^i \delta X^i$, we get 
\begin{equation}\label{eq:Aux Ito product rule}
	L_x({X_t^*}^i,{v_t^*}^i) \delta X_t^i dt = P_t^i \delta v_t^i dt - d\left(P_t^i \delta X_t^i\right) + \delta X_t^i Z_t^i dW_t.
\end{equation}
Because the process $s \mapsto \int_0^s Z_t^i \delta X_t^i dW_t$ is a martingale w.r.t. $\mathcal{F}_s$ and hence (\cite{MR2001996}, Corollary 3.2.6)
\begin{equation}\label{eq:Aux Martingale 0}
	\Ee\left[ \int_0^T Z_t^i \delta X_t^i dW_t\right]=0,
\end{equation}  
using \ref{eq:Aux Ito product rule}, the definition of $\delta \bX$, and the previous identity, we write \eqref{eq: First Variation reduced N agent min} as
\[
	\frac{1}{N} \sum_{i=1}^N \Ee\left[ \int_0^T \left( P_t^i + L_v({X_t^*}^i,{v_t^*}^i)\right)  \delta v_t^i ~ dt \right] = \frac{1}{N} \langle \bP + L_v({\bX^*},{\bv^*}) ,  \delta \bv \rangle_{\mathds{H}_{\mathds{F}}^N}=0.
\]
Hence, by Lemma \ref{lemma: Orthogonal to zero mean}, there exists $\Pi \in \mathds{H}_{\mathds{F}}$ such that for $1\leq i \leq N$
\begin{equation*}
	P^i + L_v({X^*}^i,{v^*}^i)=\Pi.
\end{equation*}
Thus, taking the mean over $i$ and using the terminal condition for $P^i$, we get
\[
	\Pi=\frac{1}{N} \sum_{i=1}^N P^i + L_v({X^*}^i,{v^*}^i), \quad \mbox{and} \quad \Pi_T = \frac{1}{N} \sum_{i=1}^N \Psi'({X_T^*}^i) + L_v({X_T^*}^i,{v_T^*}^i).\qedhere
\]	
\end{proof}
The following result shows that the existence of the price process follows from the existence of the Lagrange multiplier associated with the balance condition.

\begin{proof}[{\bf Proof of Theorem \ref{thm:MainThm}}]
By Proposition \ref{Prop: Existence N constrained}, let $\bv^*=({v^*}^1,\ldots,{v^*}^N)$ be a minimizer of \eqref{eq: reduced Nagent min problem}.
From Proposition \ref{Proposition: existence of the price}, let $\Pi \in \mathds{H}_{\mathds{F}}$ be the process that satisfies, for $1\leq i \leq N$,
\[
	P^i + L_v({X^*}^i,{v^*}^i)-\Pi=0.
\]
Hence, for $\delta \bv \in \mathds{H}_{\mathds{F}}^N$ and $1\leq i \leq N$, we have
\[
	\Ee\left[ \int_0^T \left( P^i_t + L_v({X_t^*}^i,{v_t^*}^i)-\Pi_t\right) \delta v^i_t ~ dt \right] =0.
\]
Applying It\^o's product rule to $P^i \delta X^i$, $\delta \bX$ as in \eqref{eq: deltaX}, and using \eqref{eq: N Backward P}, we rearrange \eqref{eq:Aux Ito product rule} to obtain
\begin{align*}
	d\left( P_t^i \delta X_t^i\right) &= -L_x({X_t^*}^i,{v_t^*}^i)\delta X_t^idt + Z_t^i \delta X_t^i dW_t + P_t^i \delta v_t^i dt.
\end{align*}
Hence, taking $\Ee\left[\int_0^T \cdot dt \right]$ on the previous identity, we get   
\begin{equation}\label{eq: Aux Existence price 0}
	\Ee\left[ \int_0^T d\left( P_t^i \delta X_t^i\right) + L_x({X_t^*}^i,{v_t^*}^i) \delta X_t^i dt - Z_t^i \delta X_t^i dW_t - P_t^i \delta v_t^i dt \right]=0.
\end{equation}
On the other hand, using the terminal condition for $P^i$ in \eqref{eq: N Backward P}, the initial condition for $\delta X^i$, and \eqref{eq:Aux Martingale 0}, together with \eqref{eq: N Lagrange Multiplier}, we get
\begin{align}\label{eq: Aux Existence price 1}
	&\Ee\left[ \int_0^T d\left( P_t^i \delta X_t^i\right) + L_x({X_t^*}^i,{v_t^*}^i) \delta X_t^i dt - Z_t^i \delta X_t^i dW_t  - P_t^i\delta v_t^i dt \right]
	\\
	&= \Ee\left[ \Psi'({X^*_T}^i) \delta X_T^i + \int_0^T L_x({X_t^*}^i,{v_t^*}^i) \delta X_t^i + \left(L_v({X_t^*}^i,{v_t^*}^i) - \Pi_t \right) \delta v_t^i ~ dt \right].\nonumber
\end{align}
From \eqref{eq: Aux Existence price 0} and \eqref{eq: Aux Existence price 1}, we obtain
\[
	\Ee\left[ \Psi'({X^*_T}^i) \delta X_T^i + \int_0^T L_x({X_t^*}^i,{v_t^*}^i) \delta X_t^i + \left(L_v({X_t^*}^i,{v_t^*}^i) - \Pi_t \right) \delta v_t^i ~ dt \right]=0,
\]
which is the necessary condition \eqref{eq: EL Single agent} for the optimal control ${v^*}^i$ of the agent $i$ in the representative agent problem (see Section \ref{sec: single agent problem}), with the price $\varpi$ equal to $-\Pi$.
Therefore, the minimizer $\bv^*$ of \eqref{eq: reduced Nagent min problem} defines, by Proposition \ref{Proposition: existence of the price}, the multiplier $\Pi$ such that $\bv^*$ also minimizes  
\begin{gather*}
\inf_{\bv\in \mathds{H}_{\mathds{F}}^N} \left( I_N[\bv] + \Ee\left[ \int_0^T \Pi_t \left( Q_t - \frac{1}{N}\sum_{i=1}^N v_t^i \right) dt \right] \right)
\\
 \mbox{subject to }  \quad \frac{1}{N}\sum\limits_{i=1}^N v^i=Q, ~\mbox{ and } ~ \bX ~ \mbox{ solves } ~\eqref{eq: N agents dynamics with initial condition}.
\end{gather*}
Furthermore, since $Q$ does not depend on $\bv$, we can drop the term $\Ee\left[ \int_0^T \Pi_t  Q_t dt \right]$ from the previous functional, and obtain that $\bv^*$ solves
\begin{gather*}
\inf_{\bv\in \mathds{H}_{\mathds{F}}^N} \frac{1}{N} \sum_{i=1}^N \Ee\left[\int_0^T L(X_t^i,v^i_t) - \Pi_t v_t^i ~ dt + \Psi(X^i_T) \right]
\\
\mbox{subject to }  \quad \frac{1}{N}\sum\limits_{i=1}^N v^i=Q, ~ \mbox{ and } ~ \bX ~ \mbox{ solves } ~\eqref{eq: N agents dynamics with initial condition}.
\end{gather*}
Hence, $\bv^*$ solves Problem \ref{problem: problem}  for $ \varpi = -\Pi$; that is, the multiplier $-\Pi$ of the constrained problem \eqref{eq: reduced Nagent min problem} is the price $\varpi$ of Problem \ref{problem: problem}.
Finally, under Assumption \ref{hyp: L uniformly convex}, the minimizer $\bv^*$ is unique, and hence, the multiplier $\Pi$ is uniquely defined by \eqref{eq:MultiplierFormula}, which in turn uniquely defines the price $\varpi$.
\end{proof}
%----------------------------------------------------------------------------------------------------------------------------------

\section{The linear-quadratic model}\label{Sec: LQ model}
In this section, we study the case of linear dynamics for the supply and quadratic cost structure. We consider the price formation problem for $N$ players and the representation formulas for the price obtained in Section \ref{sec: NAgent problem}. Then, we discuss the convergence as $N\to \infty$ to the limit problem for a continuum of players, which corresponds to a MFG with common noise, previously studied in \cite{GoGuRi2021}.

Let $\eta,\gamma\geq 0$, $c>0$ and $\kappa,~\zeta \in \Rr$. 
We assume the Lagrangian and the terminal cost to be
\begin{equation}\label{def:L Psi LQ}
	L(x,v)=\frac{\eta}{2}(x-\kappa)^2 + \frac{c}{2}v^2 \quad \mbox{and} \quad \Psi(x)=\frac{\gamma}{2}(x-\zeta)^2,
\end{equation}
respectively. The parameter $\zeta$ corresponds to the preferred final storage, and $\kappa$ is the preferred instantaneous storage. A natural assumption is $\zeta = \kappa$. For $\eta =0$, the running cost depends on the trading rate only. The associated Hamiltonian is
\begin{equation}\label{eq: Q Hamiltonian}
	H(x,p)=-\frac{\eta}{2}(x-\kappa)^2 + \frac{1}{2c} p^2.
\end{equation}

\subsection{Linear System formulation for finite players}
Here, we develop the analytic representation for the price, $\varpi^N$, that solves Problem \ref{problem: problem} for $N$ players. Using \eqref{eq: Q Hamiltonian}, the Hamiltonian system \eqref{eq: Hamiltonian System} for agent $i$ is
\begin{equation}\label{eq: Q Hamiltonian System} 
	\begin{cases} 
	dX^i_t = -\frac{1}{c}(P^i_t+\varpi_t) dt
	\\
	X^i_0=x^i_0
	\\
	dP^i_t =  - \eta (X^i_t-\kappa) dt + Z^i_t dW_t
	\\
	P^i_T=\gamma(X^i_T-\zeta),
	\end{cases}
\end{equation}
and the optimal control (see Proposition \ref{proposition: Hamilton System}) simplifies to $v^i=-\frac{1}{c}(P^i+\varpi)$.
From Proposition \ref{Proposition: existence of the price}, the price has the formula
\begin{gather}\label{eq: price-adjoint var relation}
	\varpi^N =-\frac{1}{N}\sum_{i=1}^N (P^i + c v^i) =-\left(\frac{1}{N}\sum_{i=1}^N P^i+c Q\right).
\end{gather}
Assuming that $Q$ is described by an It\^{o} differential, we take differentials in the previous and using \eqref{eq: Q Hamiltonian System}, we see that
\begin{equation}\label{eq:priceN players formula}
	d\varpi^N = \eta \left( \overline{X}_t - \kappa \right) dt -\overline{Z}_t dW_t - c dQ,	
\end{equation}
where 
\[
	\overline{X} = \frac{1}{N}\sum_{i=1}^N X^i, \quad \mbox{and }\quad  	\overline{Z} = \frac{1}{N}\sum_{i=1}^N Z^i.
\]
From the balance condition \eqref{eq:Balance condition}, we have that $d\overline{X}_t = Q_t dt$; that is,
\begin{equation*}%\label{eq:Xbar N Players}
	\overline{X}_t = \overline{x}_0 + \int_0^T Q_s ds, \quad t \in [0,T],	
\end{equation*}
where $\overline{x}_0$ is the mean of the initial positions $x_0^i$ of the agents. Therefore, we obtain a representation formula for the dynamics of $\varpi^N$ once the It\^{o} dynamics of $Q$ are given. Yet, this representation involves the processes $Z^i$. To gain insight into the computation of the process $\overline{Z}$, we eliminate the dependence of \eqref{eq: Q Hamiltonian System} on the price using \eqref{eq: price-adjoint var relation}. We obtain
\begin{equation*}
	\begin{cases} 
	dX^i_t = -\frac{1}{c N}\sum\limits_{j=1}^N (P^i_t-P^j_t) +Q_t ~ dt
	\\
	X_0^i=x_0^i
	\\
	dP^i_t =  - \eta (X^i_t - \kappa) dt + Z^i_t dW_t
	\\
	P^i_T=\gamma(X^i_T-\zeta),
	\end{cases}
\end{equation*}
which corresponds to the following linear system for the $N$ players 
\[
	\begin{cases} 
	d\bX_t = (B \bP_t +Q_t \mathds{1}_N) ~ dt
	\\
	\bX_0=\bx_0
	\\
	d\bP_t =  - \eta( \bX_t - \kappa \mathds{1}_N )dt + \bZ_t dW_t
	\\
	\bP_T=\gamma(\bX_T-\zeta\mathds{1}_N),
	\end{cases}
\]
where
\[
	 \quad B=\frac{1}{c N}\begin{bmatrix}
	1-N & 1 & \ldots & 1\\
	1 & 1-N & \ldots & 1\\
	\vdots & \vdots & \ddots & \vdots\\
	1 & 1 & \ldots & 1-N
	\end{bmatrix}, \quad \mathds{1}_N=\begin{bmatrix}
	1 \\
	1\\
	\vdots\\
	1
	\end{bmatrix}.
\]
The previous is a linear forward-backward SDE system. For the solvability of such systems, two main approaches have been proposed: the Four Step Scheme and the Method of Continuation (see \cite{FBSDEbook2007}, Chapters 4 and 6). In the former, the coefficients are required to be deterministic, which is not the case due to the dependence on $Q$. The latter admits systems with random coefficients, but the method relies on the existence of a so-called bridge, which transforms the given system into one whose solution is required to be known. Moreover, the construction of such bridges has proven to be useful for one-dimensional problems, but it is not trivial for high-dimensional systems. Other techniques to reduce the previous system include the variation of constants formula for $(\bX,\bP)$ in terms of the process $\bZ$, and the use of Riccati-type equations (see \cite{FBSDEbook2007}, Chapter 2). The reduction techniques have no trivial extension to the case of random coefficients (see \cite{YongRandomCoeff}).

Alternatively, we can consider the dynamics of the mean processes $\overline{X}$ and $\overline{Z}$, which, according to \eqref{eq: Q Hamiltonian System} and \eqref{eq: price-adjoint var relation}, follow
\[
	\begin{cases}
	d\overline{X}_t = Q_t dt
	\\
	\overline{X}_0=\overline{x}_0
	\\
	dQ_t=b^S(Q_t,t)dt+ \sigma^S(Q_t,t)dW_t
	\\
	Q_0=q_0
	\\
	d\overline{P}_t = - \eta \left(\overline{X}_t-\kappa \right) dt + \overline{Z}_t dW_t
	\\
	\overline{P}_T = \gamma \left( \overline{X}_T - \zeta\right).
	\end{cases}
\]
Following the standing assumption in the formulation of the Four Step Scheme, we assume that $\overline{P}=\theta\left( \overline{X},Q,t\right)$ for some $\theta: \Rr^2 \times [0,T] \to \Rr$. Then, we look for a parabolic PDE characterizing $\theta$ by considering the relation between the drift and volatility in the previous system derived from It\^{o} formula applied to $\theta(\overline{X},Q,T)$. A difficulty is the degeneracy of the forward component $(\overline{X},Q)$ because it does not depend on $\overline{Z}$. Therefore, we can not recover a consistency condition that completely determines the function $\theta$.  Yet, \eqref{eq:priceN players formula} provides a useful representation to study the limit as $N\to \infty$. In Section \ref{sec:numerical}, we consider a discrete representation of the noise to approximate numerically the price $\varpi^N$ that solves the $N$ players game. 

\subsection{Optimal control formulation for infinite players}

Here, we adopt an optimal control approach in an extended state space to compute the price, $\varpi^{\infty}$, that solves the analogous of Problem \ref{problem: problem} for a continuum of players. We obtain explicit formulas for the price up to the solution of an ODE system. Using the explicit representation, we consider the convergence of $\varpi^N$ to $\varpi^\infty$ as $N\to \infty$. 

The linear-quadratic price formation MFG problem was studied in \cite{GoGuRi2021}. Here, we focus on the explicit solution representation for the linear-quadratic case for $\eta \neq 0$. Assume the supply $Q$ follows the SDE
\begin{equation}\label{eq: supply dynamics}
dQ_t=b^S(Q_t,t)dt+\sigma^S(Q_t,t)dW_t,
\end{equation}
where $b^S:\Rr\times[0,T] \to \Rr$ is the drift and $\sigma^S:\Rr\times[0,T] \to \Rr$ is the volatility, which are measurable smooth functions that satisfy
\begin{equation*}
%\label{eq:Lipschitz cond supply}
\begin{aligned}
|b^S(q,t)-b^S(p,t)|+|\sigma^S(q,t)-\sigma^S(p,t)| \leq C |q-p|\\
|b^S(q,t)|+|\sigma^S(q,t)| \leq D(1+|q|)
\end{aligned}
\quad\quad\mbox{for all } q \in \Rr,~ t\in [0,T]
\end{equation*}
for some constants $C,D>0$.
These conditions guarantee the existence of $Q$ (see \cite{MR2001996}, Theorem 5.2.1 for further details). The standing assumption is that the price follows
\[
	d\varpi^\infty_t = b^P(Q_t,\varpi^\infty_t,t) dt + \sigma^P(Q_t,\varpi^\infty_t,t)dW_t,
\]
where the drift $b^P$, the volatility $\sigma^P$, and the initial condition $w_0 \in \Rr$ are to be determined. 
%the MFG problem corresponds to find $u:\Rr^3 \times [0,T] \to \Rr$, $m \in C([0,T] \times \Omega; \Pp(\Rr^3))$, $w_0 \in \Rr$, and $b^P,\sigma^P: \Rr^2 \times [0,T] \to \Rr$ solving
%\[
%	\begin{cases}
%	-u_t + \tfrac{1}{2c}(w+u_x)^2 = b^S u_q + b^P u_w + \tfrac{1}{2}(\sigma^S)^2 u_{qq} + \sigma^S \sigma^P u_{qw} + \tfrac{1}{2}(\sigma^P)^2 u_{ww}
%	\\
%	u(x,q,w,T)=\tfrac{\gamma}{2}(x-\zeta)^2
%	\\
%	dm_t = \left(\left(\tfrac{m (\sigma^S)^2}{2}\right)_{qq} + \left( m \sigma^S \sigma^P \right)_{qw} + \left(\tfrac{m (\sigma^P)^2}{2}\right)_{ww} - \div(m \bb) \right)dt - \div(m \bsigma)dW_t
%	\\
%	m_0 = \tilde{m}_0 \times \delta_{q_0} \times \delta_{w_0}
%	\\
%	\int_{\Rr^3} q + \tfrac{1}{c}(w+u_x(x,q,w,t))m_t(dx\times dq \times dw) =0,~ 0\leq t \leq T,
%	\end{cases}
%\]	
%where $\bb=\left(-\tfrac{1}{c}(w+u_s),b^S,b^P\right)$, $\bsigma = (0,\sigma^S,\sigma^P)$, and the divergence is w.r.t. $(x,q,w)$.
The approach presented in \cite{GoGuRi2021} considered the case $\eta=0$. Here, we extend that approach to include the case $\eta \neq 0$. 
In this setting, we can characterize the price as the unique solution of an SDE.
Let $x_0,\overline{x}_0,q_0 \in\Rr$.
We consider the following dynamics
\begin{equation}\label{eq: LQ Augmented dynamics}
	\begin{cases} 
	d X_t = v_t ~ dt
	\\
	X_0=x_0
	\\
	d \overline{X}_t^\infty = Q_t ~ dt
	\\
	\overline{X}_0^\infty=\overline{x}_0,
	\\
	dQ_t =	b^S(Q_t,t)dt+\sigma^S(Q_t,t)dW_t
	\\
	Q_0=q_0
	\\
	d\varpi_t^\infty =	b^P(X_t,\overline{X}_t ,Q_t,\varpi_t,t)dt+\sigma^P(X_t,\overline{X}_t,Q_t,\varpi_t,t)dW_t
	\\
	\varpi_0^\infty=w_0.
	\end{cases}
\end{equation}
In the previous, $w_0\in \Rr$ and the coefficients $b^P$ and $\sigma^P$ are unknown.
We make the key assumption that the coefficients of the SDE driving the price have the form in \eqref{eq: LQ Augmented dynamics}.
As we will see in \eqref{eq: Q Coeffiecients relation price supply}, this is the case if the supply's coefficients $b^S$ and $\sigma^S$ are linear.

From the standard optimal control theory, define the value function $u:\Rr^4 \times [0,T] \to \Rr$ by
\begin{equation*}
u(x,\overline{x},q,w,t)=\inf_{v\in L^2([t,T]\times\Omega)} \Ee\left[\int_t^T L(X_s,v_s) + \varpi_s v_s ~ ds + \Psi(X_T) \right],
\end{equation*}
where $(X,\overline{X},Q,\varpi)$ solves \eqref{eq: LQ Augmented dynamics} for $t \leq s \leq T$ and initial condition $(x,\overline{x},q,w)$ at $t$. 
%\begin{equation*}
%	\begin{cases} 
%	d X_s = v_s ~ ds
%	\\
%	X_t=x
%	\\
%	d \overline{X}_s = Q_s ~ ds
%	\\
%	\overline{X}_t=\overline{x}
%	\\
%	dQ_s =	b^S(Q_s,s)ds+\sigma^S(Q_s,s)dW_s
%	\\
%	Q_t=q
%	\\
%	d\varpi_s =	b^P(X_s,\overline{X}_s ,Q_s,\varpi_s,s)ds+\sigma^P(X_s,\overline{X}_s ,Q_s,\varpi_s,s)dW_s
%	\\
%	\varpi_t=w.
%	\end{cases}
%\end{equation*}
The corresponding Hamilton-Jacobi-Bellman equation is
\begin{equation}\label{eq: HJB quadratic}
	\begin{cases} 
	-u_t + H(x,w+u_x) = q u_{\overline{x}} + b^S u_q + b^P u_w + \frac{1}{2}(\sigma^S)^2 u_{qq}+\sigma^S \sigma^P u_{qw} +\frac{1}{2}(\sigma^P)^2 u_{ww}
 	\\
	u_T=\Psi,
	\end{cases}
\end{equation}
where all functions are evaluated at $(x,\overline{x},q,w,t)$.
Whenever $u$ is smooth enough, the optimal control in feedback form is
\[
	v^*(s) = -H_p(X_s,\varpi_s + u_x(X_s,\overline{X}_s,Q_s,\varpi_s,s))=-\frac{1}{c}(\varpi_s + u_x(X_s,\overline{X}_s,Q_s,\varpi_s,s)).
\]
Given $\tilde{m}_0 \in \Pp(\Rr)$, the balance condition corresponds to
\begin{equation}\label{eq:Balance condition infty}
	\int_{\Rr} -\frac{1}{c}(\varpi_t + u_x(X_t,\overline{X}_t,Q_t,\varpi_t,t))\tilde{m}_0(x)dx = Q_t, \quad 0\leq t \leq T,
\end{equation}
where $(X,\overline{X}^\infty,Q,\varpi)$ is the solution of \eqref{eq: LQ Augmented dynamics} with initial condition $(x,\mu_0,q_0,w_0)$, where $\mu_0$ denotes the mean of $\tilde{m}_0$. Under linear dynamics, the coefficients $b^P$ and $\sigma^P$ in \eqref{eq: HJB quadratic} have an explicit representation, as we show next.
\subsubsection{Linear dynamics and quadratic solutions}%\label{subsec LQ}
We further assume the dynamics of the supply have a linear structure
\begin{equation}\label{eq:Q linear dynamics}
	dQ_t = \left(b^S_{1}(t) Q_t +  b^S_{0}(t)\right)dt +\left(\sigma^S_{1}(t) Q_t +  \sigma^S_{0}(t)\right)dW_t.
\end{equation}
Hence, assume that $u$ is a second-degree polynomial in $x,\overline{x},q$ and $w$; that is,
\begin{align}\label{eq: Q u statement}
	u(x,\overline{x},q,w,t)=&a_0(t)+a_1^1(t) x+a_1^2(t) \overline{x}+ a_1^3(t) q + a_1^4(t) w
	\\
	 &+ a_2^1(t) x^2+ a_2^2(t) x\overline{x} + a_2^3(t) xq + a_2^4 (t) xw +  a_2^5(t) \overline{x}^2 + a_2^6(t) \overline{x}q + a_2^7(t) \overline{x}w \nonumber
	 \\
	 & + a_2^8(t) q^2 + a_2^9(t) qw + a_2^{10}(t) w^2, \nonumber
\end{align}
where $a_i^j:[0,T]\to \Rr$. 
Differentiating \eqref{eq: HJB quadratic} w.r.t. $x$ and applying the It\^o differential rule to the balance condition \eqref{eq:Balance condition infty}, 
%\[
%	du_x(X^i_t, \overline{X}_t, Q_t, \varpi_t,t) = - \eta (X_t^i - \kappa) dt + (u_{xq} \sigma^S + u_{xw} \sigma^P) dW_t.
%\]
%We use the previous identity to compute the differential of \eqref{eq: price-adjoint var relation}; that is,
%\begin{align*}
%	d \varpi_t &= - \left( \frac{1}{N} \sum_{i=1}^N du_x(X^i_t, \overline{X}_t, Q_t, \varpi_t,t) + c dQ_t \right)
%	\\
%		& = \left( \eta (\overline{X}_t - \kappa) - c b^S\right) dt - \left( (u_{xq} + c)\sigma^S + u_{xw}  \sigma^P \right) dW_t
%	\\
%		& = \left( \eta (\overline{X}_t - \kappa) - c b^S\right) dt - \left( (a_2^3 + c)\sigma^S + a_2^4  \sigma^P \right) dW_t.
%\end{align*}
we obtain that the drift and the volatility in \eqref{eq: HJB quadratic} are
\begin{align}\label{eq: Q Coeffiecients relation price supply}
	\begin{split}
	b^P(x,\overline{x},q,w,t) &=  \eta \left( \overline{x} - \kappa\right) - c b^S(q,t),
	\\
	\sigma^P(x,\overline{x},q,w,t) &= -\frac{a_2^3(t) + c}{a_2^4(t) +1}\sigma^S(q,t).
	\end{split}
\end{align}

The previous coefficients exhibit fundamental properties of the quadratic cost structure \eqref{def:L Psi LQ}. For instance, one term in the drift is proportional to the difference between the time-average supply, represented by $\overline{x}$, and the preferred running state $\kappa$, and the second term in the drift is the opposite behavior of the supply dynamics, proportional to the coefficient $c$ in the running cost. For the volatility, we observe a linear dependence on the supply's volatility, proportional to the running cost. We observe that the supply dynamics entirely determined the price dynamics. For instance, assuming mean-reverting dynamics for the supply
\begin{equation}\label{eq:Q SDE mean reverting}
	dQ_t = \left(\overline{Q}(t) - Q_t \right)dt +\sigma_s dW_t,
\end{equation}
where $\overline{Q}:[0,T] \to \Rr$ and $\sigma_s \in \Rr$, and replacing \eqref{eq: Q Coeffiecients relation price supply} and \eqref{eq: Q u statement} in \eqref{eq: HJB quadratic}, we obtain the following ODE system for the $a_i^j$ functions
{\footnotesize
\begin{equation*}
\begin{aligned}[c]
\dot{a}_0 &=a_1^4 (c \overline{Q}+\eta  \kappa )-\overline{Q} a_1^3+\frac{\sigma_s^2 a_2^9 \left(a_2^3+c\right)}{a_2^4+1}
\\ 
& \quad -\frac{\sigma_s^2 a_2^{10} \left(a_2^3+c\right)^2}{\left(a_2^4+1\right)^2}+\frac{\left(a_1^1\right)^2}{2 c}-\sigma_s^2 a_2^8-\frac{\eta  \kappa ^2}{2}
\\
\dot{a}_1^1 &=\overline{Q} \left(c a_2^4-a_2^3\right)+\frac{2 a_1^1 a_2^1}{c}+\eta  \kappa  \left(a_2^4+1\right)
\\ 
\dot{a}_1^2 &=c \overline{Q} a_2^7-\overline{Q} a_2^6+\frac{a_1^1 a_2^2}{c}+\eta  \kappa  a_2^7-\eta  a_1^4
\\ 
\dot{a}_1^3 &=c \overline{Q} a_2^9-2 \overline{Q} a_2^8-c a_1^4+\frac{a_1^1 a_2^3}{c}+\eta  \kappa  a_2^9-a_1^2+a_1^3
\\ 
\dot{a}_1^4 &=-\overline{Q} \left(a_2^9-2 c a_2^{10}\right)+\frac{a_1^1 \left(a_2^4+1\right)}{c}+2 \eta  \kappa  a_2^{10}
\\ 
\dot{a}_2^1 &=\frac{2 \left(a_2^1\right)^2}{c}-\frac{\eta }{2}
\\ 
\dot{a}_2^2 &=\frac{2 a_2^1 a_2^2}{c}-\eta  a_2^4
\end{aligned}
\hskip0.3cm
\begin{aligned}[c] 
\dot{a}_2^3 &=\frac{a_2^3 \left(2 a_2^1+c\right)}{c}-c a_2^4-a_2^2
\\
\dot{a}_2^4 &=\frac{2 a_2^1 \left(a_2^4+1\right)}{c}
\\ 
\dot{a}_2^5 &=\frac{\left(a_2^2\right)^2}{2 c}-\eta  a_2^7
\\ 
\dot{a}_2^6 &=\frac{a_2^2 a_2^3}{c}-c a_2^7-\eta  a_2^9-2 a_2^5+a_2^6
\\ 
\dot{a}_2^7 &=\frac{a_2^2 \left(a_2^4+1\right)}{c}-2 \eta  a_2^{10}
\\ 
\dot{a}_2^8 &=\frac{\left(a_2^3\right)^2}{2 c}-c a_2^9-a_2^6+2 a_2^8
\\ 
\dot{a}_2^9 &=\frac{a_2^3 \left(a_2^4+1\right)}{c}-2 c a_2^{10}-a_2^7+a_2^9
\\ 
\dot{a}_2^{10} &=\frac{\left(a_2^4+1\right)^2}{2 c}
\end{aligned}
\end{equation*}
}with the terminal conditions $a_ 0(T)=\frac{\gamma  \zeta ^2}{2}$, $a_1^1(T)=-\gamma  \zeta$, $a_2^1(T)=\frac{\gamma }{2}$, and zero for all other variables.

Hence, the price $\varpi^\infty$ is obtained as part of the solution to the following SDE system
\begin{equation}\label{eq:System price LQ}
	\begin{cases} 
	d \overline{X}_t^{\infty} = Q_t ~ dt
	\\
	\overline{X}_0^{\infty}=\mu_0,
	\\
	dQ_t =	(\overline{Q}(t)-Q_t) dt +\sigma_s dW_t
	\\
	Q_0=q_0
	\\
	d\varpi_t^\infty =	\left(\eta(\overline{X}_t^{\infty} - \kappa) - c (\overline{Q}(t)-Q_t)\right) dt  - \frac{a_2^3(t) + c}{a_2^4(t) +1} \sigma_s dW_t
	\\
	\varpi_0^\infty=w_0,
	\end{cases}
\end{equation}
where the initial condition for the price, $w_0$, is given by \eqref{eq:Balance condition infty} as
%{\color{blue}
%\begin{gather*}
%\varpi_0 =-\left(\frac{1}{N}\sum_{i=1}^N P^i_0+\beta Q_0\right)
%\\ =-\left(\frac{1}{N}\sum_{i=1}^N u_x(X^i_0,\overline{X}_0,Q_0,\varpi_0,0) +\beta Q_0\right)	 
%\\=-\left(a_1^1(0)+(2  a_2^1(0) + a_2^2(0))\overline{x}_0 + (a_2^3(0) +\beta ) q_0 + \varpi_0 a_2^4(0) \right)
%\end{gather*}
%}
\begin{gather}\label{eq:Initial price LQ}
	w_0 =-\frac{\mu_0 \left(2 a_ 2^1(0)+a_ 2^2(0)\right)+q_0\left(  a_ 2^3(0)+c \right) +a_ 1^1(0)}{a_ 2^4(0)+1}.
\end{gather}
The initial price relates linearly to the initial density, with a coefficient that depends implicitly on the parameters $\eta$, $\gamma$ and $c$,  and linearly to the initial supply, with an explicit coefficient $c$, inherited from the running cost. In this case, the functions $a_2^1$,$a_2^2$,$a_2^3$, and $a_2^4$ form a sub-system of ODEs that is independent of the other $a_i^j$ functions.
This sub-system has the analytic solutions
{\footnotesize
%\begin{align*}
%	a_2^3(t)&=\sqrt{\tfrac{c\eta}{c\eta - \gamma^2}}	\left[\left( \left(e^{t-T}-1\right)
%  \gamma +c+\eta 
%   \left(T-t-1+e^{t-T}\right)\right)+\sqrt{\tfrac{c}{\eta}} \gamma   \sinh
%   \left(\tfrac{\sqrt{\eta } (t-T)}{\sqrt{c}}\right) \right.
%   \\
%  & \quad  \left.- c  \cosh
%   \left(\tfrac{\sqrt{\eta } (t-T)}{\sqrt{c}}\right)\right] \text{sech}\left(\sqrt{\tfrac{\eta}{c}}
%   (t-T)-\tanh ^{-1}\left(\tfrac{\gamma }{\sqrt{c\eta }}\right)\right)   
%\end{align*}
\begin{align*}
   a_2^1(t)&= \tfrac{\sqrt{c\eta }}{2} \tanh \left(\tanh ^{-1}\left(\tfrac{\gamma
   }{\sqrt{c\eta }}\right)+\sqrt{\tfrac{\eta}{c} } (T-t)\right)
	\\
		a_2^4(t)&=\sqrt{\tfrac{c\eta}{c\eta - \gamma^2}}\text{sech}\left(\sqrt{\tfrac{\eta}{c}}(t-T)-\tanh
   ^{-1}\left(\tfrac{\gamma }{\sqrt{c\eta }}\right)\right)-1
   \\
	a_2^2(t)&= \left[\sqrt{c \eta } \sinh \left(\sqrt{\tfrac{\eta}{c} }
   (t-T)\right)+\eta( T-t)-\gamma-\gamma  \cosh \left(\sqrt{\tfrac{\eta}{c} }
   (t-T)\right)\right]\left(a_2^4(t)+1\right)
  	\\
   	a_2^3(t)&=	\left[\left( \left(e^{t-T}-1\right)
  \gamma +c+\eta 
   \left(T-t-1+e^{t-T}\right)\right)+\sqrt{\tfrac{c}{\eta}} \gamma   \sinh
   \left(\sqrt{\tfrac{\eta}{c}}(t-T)\right) \right.
   \\
  & \quad  \left.- c  \cosh
   \left(\sqrt{\tfrac{\eta}{c} } (t-T)\right)\right]  \left(a_2^4(t)+1\right),
\end{align*}
}for $\eta>0$ and $c\eta - \gamma^2 >0$, and
{\footnotesize
\begin{equation*}
   a_2^1(t)= \frac{c \gamma }{2 c-2 \gamma  t+2 \gamma  T}, \quad
   a_2^4(t)=\frac{\gamma  (t-T)}{c+\gamma  (T-t)}, \quad
   a_2^2(t)= 0, \quad
   a_2^3(t)=	-\frac{c \gamma \left( T-t-1+e^{t-T}\right)}{c+\gamma  (T-t)},
\end{equation*}
}for $\eta=0$. 

Notice that the right-hand side of the SDE for the price in \eqref{eq:System price LQ} does not include $\varpi$. Therefore, using the previous formulas, the price is explicitly given in \eqref{eq:System price LQ}-\eqref{eq:Initial price LQ} by the initial conditions $\tilde{m}_0 \in \Pp(\Rr)$, $q_0\in \Rr$, the supply process $Q$, and the parameters $T,~\eta,~\gamma,~c,~\kappa$, and $\zeta$. Moreover, we can compute measures of variability between price and supply, such as the covariance
\begin{align*}
	\mbox{Cov}\left(Q_t , \varpi_t^\infty\right) = & -\tfrac{\sigma_s^2}{2}  \left(e^t-1\right) e^{-2 t-T} \left(c \left(e^t+1\right)
   e^T+\gamma  e^t \left(e^t-2 e^T+1\right) \right.
   \\
   & \left. +\eta  \left(e^{t+T} (-2 t+2 T-1)+e^t+e^{2
   t}-e^T\right)\right),
\end{align*}
for $\eta>0$ and $c\eta - \gamma^2 >0$, and
\begin{align}\label{eq:LQCorrelation0}
	\mbox{Cov}\left(Q_t , \varpi_t^\infty\right) = - \tfrac{\sigma_s^2}{2}  \left(e^t-1\right) e^{-2 t-T} \left(c\left(e^t+1\right) e^T-\gamma  e^t \left(e^t-2 e^T+1\right)\right)
\end{align}
for $\eta=0$. The previous formulas verify that the intuitive negative correlation between price and supply holds in our model. For instance, in the case $\eta =0$ , from \eqref{eq:LQCorrelation0} we have
 \[
	\tfrac{d^2}{dt^2}\mbox{Cov}\left(Q_t , \varpi_t^\infty\right) = \tfrac{\sigma_s^2}{2}  e^{-2 t-T} \left(4 c e^{T} +2 \gamma e^{T+t} + \gamma e^t\left( e^{2t}-1\right)\right) \geq 0,
\]
$\mbox{Cov}\left(Q_0 , \varpi_0^\infty\right) = 0$, and $\mbox{Cov}\left(Q_T , \varpi_T^\infty\right)  = -\tfrac{\sigma_s^2}{2} c \left( 1 - e^{-2T} \right)<0$; that is, \eqref{eq:LQCorrelation0} is a convex function which is $0$ at $t=0$, negative at $t=T$, and thus negative on $(0,T]$. We use the previous measures of joint variability between supply and price in Section \ref{sec:numerical}.

%{\tiny
%\[
%\dot{a}_1^1(t)=\frac{\sqrt{\eta } \left(\bar{Q} \left(c^{3/2} 
%   \left(e^T-e^t\right)\tfrac{\gamma}{\sqrt{c}}-c \eta  \left(e^T (-t+T-1)+e^t\right)\right)+e^T \left(a_1^1(t) \left(\gamma  \cosh
%   \left(\frac{\sqrt{\eta } (t-T)}{\sqrt{c}}\right)-\sqrt{c} \sqrt{\eta } \sinh
%   \left(\frac{\sqrt{\eta } (t-T)}{\sqrt{c}}\right)\right)+c \eta  \kappa \right)\right)}{c
%   \sqrt{\eta } e^T \cosh \left(\frac{\sqrt{\eta } (t-T)}{\sqrt{c}}\right)-\sqrt{c} \gamma  e^T
%   \sinh \left(\frac{\sqrt{\eta } (t-T)}{\sqrt{c}}\right)}
%\]}
\subsubsection{Convergence of the finite game to the continuum game} For the linear-quadratic structure, \eqref{eq: price-adjoint var relation}, \eqref{eq:priceN players formula} and \eqref{eq: supply dynamics} show that $\varpi^N$ is given by the SDE system
\begin{equation}\label{eq:SDE N players}
	\begin{cases} 
	d \overline{X}_t = Q_t ~ dt
	\\
	\overline{X}_0=\overline{x}_0
	\\
	dQ_t =	b^S(Q_t,t) dt +\sigma^S(Q_t,t) dW_t
	\\
	Q_0=q_0
	\\
	d\varpi_t^N =	\left(\eta(\overline{X}_t - \kappa) - c b^S(Q_t,t) \right) dt  - \left(\overline{Z}_t + c \sigma^S(Q_t,t)\right) dW_t
	\\
	\varpi_0^N=-(\overline{P}_0 + cq_0),
	\end{cases}
\end{equation}
and, by \eqref{eq: Q Coeffiecients relation price supply}, $\varpi^\infty$ is given by the SDE system
\begin{equation*}
	\begin{cases} 
	d \overline{X}_t^\infty = Q_t ~ dt
	\\
	\overline{X}_0^\infty=\mu_0
	\\
	dQ_t =	b^S(Q_t,t) dt +\sigma^S(Q_t,t) dW_t
	\\
	Q_0=q_0
	\\
	d\varpi_t^\infty =	\left(\eta(\overline{X}_t^\infty - \kappa) - c b^S(Q_t,t) \right) dt  - \tfrac{a_2^3(t)+c}{a_2^4(t)+1}\sigma^S(Q_t,t) dW_t
	\\
	\varpi_0^\infty =-\left(\mu_0 \left(2 a_ 2^1(0)+a_ 2^2(0)\right)+q_0\left(  a_ 2^3(0)+c \right) +a_ 1^1(0)\right)(a_ 2^4(0)+1)^{-1}.
	\end{cases}
\end{equation*}
The previous two systems show that the convergence of $\varpi^N$ to $\varpi^\infty$ as $N\to \infty$ in $\mathbb{H}_{\mathbb{F}}$ (which corresponds to having two It\^{o} processes described by the same SDE) relies on the convergence of $\overline{x}_0$ to $\mu_0$ as $N \to \infty$, which is guaranteed by the law of large numbers when the initial states of the $N$ players, $x_0^i$, are sampled independently and with identical distribution $\tilde{m}_0$. 

%For a general cost and dynamics structure, we require a careful examination of the convergence of the random empirical measure
%\[
%	\hat{m}^N_t := \frac{1}{N} \sum_{i=1}^N \delta_{X_t^i}
%\]
%to the random measure $m^{\infty}$ on $\Rr^4 \times \Omega$ defined by
%\begin{align*}
%	& \int_{\Rr^4} f(x,\overline{x},q,w) m^{\infty}_t (dx\times d\overline{x} \times dq \times dw) 
%	\\
%	& = \int_{\Rr^4} f(X_t,\overline{X}^{\infty}_t,Q_t,\varpi_t^\infty) \tilde{m}_0(dx) \times \delta_{\mu_0}(d\overline{x}) \times \delta_{q_0}(dq) \times \delta_{w_0^\infty}(dw),
%\end{align*}
%for $f \in C_b(\Rr^4)$, and where $(X,\overline{X}^\infty,Q,\varpi^\infty)$ satisfies \eqref{eq: LQ Augmented dynamics}. A main challenge is that the previous are not probability measures but random probability measures due to the underlying dependence on the common noise. Therefore, the use of Wasserstein metrics (which is employed in the deterministic setting) to describe a notion of convergence becomes more intricate due to the lack of compactness of the space of random measures (see \cite{carmona2018probabilistic}).

\section{Numerical Results and Real Data}\label{sec:numerical}

Here, we address the numerical computation of the price both for the finite and the continuum number of players. In the finite case, we discretize  the minimization problem \eqref{eq: Nagent min problem} using a Binomial Tree representation of the noise. The computation of the price reduces to a finite high-dimensional optimization problem. We illustrate this method with the linear-quadratic model of Section \ref{Sec: LQ model}, and we show the convergence, as the number of players grows, to the solution of the continuum model. Then, we specialize the models to simulate the price obtained using real data from the electricity grid in Spain.

\subsection{Numerical approximation of the finite players model} In this section, we numerically approximate the price $\varpi^N$ solving Problem \ref{problem: problem} for $N$ players using a discrete approximation of the minimization problem \eqref{eq: Nagent min problem}. Our formulation admits a general structure on the supply dynamics and cost functions, including the linear-quadratic model of Section \ref{Sec: LQ model} as a particular case. Our approach relies on a discrete representation of the common noise using a Binomial Tree.

\subsubsection{Binomial Tree approximation}

In our model, the common noise corresponds to the Brownian Motion in \eqref{eq: supply dynamics}, which specifies the supply dynamics. Thus, every realization of the Brownian motion path determines a realization for both supply and price. For instance, \eqref{eq:SDE N players} provides the supply and price paths for any realization of the noise, which is a feature of the linear-quadratic model. However, for general dynamics on the supply and non-quadratic cost, even if the supply process can be exactly simulated, there is no guarantee that the price process can be explicitly solved. Therefore, we consider a finite-dimensional approximation of the noise process. This implies that both supply and price become finite-dimensional objects as well. The advantage of this numerical approach is that our model becomes a finite-dimensional convex optimization problem, which can be solved using standard methods. We adopt a Binomial Tree representation of the Brownian motion. The convergence results for schemes similar to the one presented here are studied in \cite{Touzi2013},  Chapter 12.

Let $T>0$ be the time horizon and $M\in\mathbb{N}$ be the number of time steps.
Let $h=T/M$, and $t_k = kh$ for $k=0,\ldots,M$.
We use the Forward-Euler discretization for the supply 
\begin{equation}\label{eq:Qdiscrete Forward Euler}
Q_{k+1}=Q_{k}+b^S(Q_k,k)h+\sigma^S(Q_k,k)\Delta W_k, \quad k=0,\ldots,M-1,
\end{equation}
where $\Delta W_0=0$, and $\Delta W_k$, for $k=1,\ldots,{M-1}$, are the discrete approximation of the Brownian motion.
We select $\Delta W_k=\sqrt{h}\xi_k$, where $\xi_k$ are i.i.d.
(binomial) random variables taking the values $\pm 1$ with the same probability. % (see Figure \ref{image:BTDiagram}).
Hence, at time level $k$, $Q_k\in \left\{ Q_{1,k},\ldots,Q_{2^k,k}\right\}$ (see Figure \ref{image:BTDiagram Q}).
The discrete $\sigma$-algebras are $\mathcal{F}_0=\left\{\emptyset,\Omega\right\}$, and $\mathcal{F}_k=\sigma\left(\Delta W_{j}:~ 0\leq j \leq k\right)$ for $k=1,\ldots,M$.
Let $\mathbf{v}^i=(\mathrm{v}^i_0,\ldots,\mathrm{v}^i_{M-1})$ denote the discrete approximation of the control for agent $i$ obtained from the Binomial Tree.
The measurability condition w.r.t.
$\mathcal{F}_k$ means that $\mathrm{v}^i_k \in \left\{ \mathrm{v}^i_{1,k},\ldots,\mathrm{v}^i_{2^k,k}\right\}$ for $0 \leq k \leq M-1$, where the variables $\mathrm{v}^i_{j,k}$ are the decision variables for the discrete optimization problem.
Notice that at time level $k$, the expectation operator becomes an average over $2^k$ values.
We compute $X^i_{k+1}$, the position of the agent $i$ at time $t_{k+1}$, using the Forward-Euler formula in \eqref{eq: Agent dynamics}; that is,
\[
	X^i_{k+1}=X^i_{k}+h \mathrm{v}^i_{k}  , \quad k=0,\ldots,M-1,
\] 
where $X^i_0=x_0^i$.
Because the initial condition $\bx_0=(x_0^1,\ldots,x_0^N)\in \Rr^N$ is given, the positions $X^i_{k+1}$, for $1\leq i \leq N$ and $0\leq k \leq M-1$, depend only on the velocity variables.
\begin{remark}
Because the random variables $\xi_k$ are binomial, the discrete noise process $\Delta W$ has $2^{M}$ realizations. %, as illustrated in Figure \ref{image:BTDiagram} for $M=2$ time steps.
Accordingly, as shown in Figure \ref{image:BTDiagram Q}, each realization of the noise process determines one realization of the supply process. For ease of notation, we do not index the realization to which the variable $Q_{j,k}$ corresponds.
Likewise, we denote by $X^i_{j,k}$ the position of agent $i$ at time level $k$ computed using the velocity variable $\mathrm{v}^i_{j,k}$, where both variables correspond to the same realization of the noise.

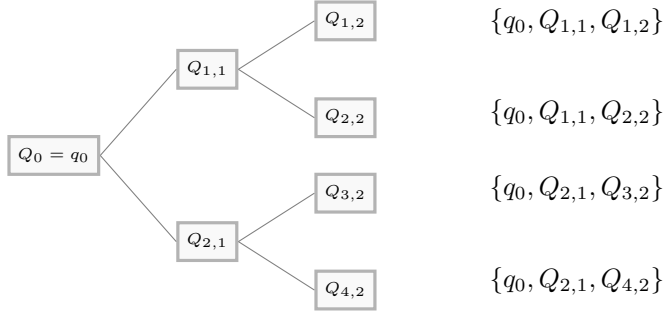
\begin{figure}[ht]
\centering
\begin{tikzpicture}[
roundnode/.style={circle, draw=gray!60, fill=gray!5, very thick, minimum size=6mm},
squarednode/.style={rectangle, draw=gray!60, fill=gray!5, very thick, minimum size=5mm},
]
	\node[squarednode] (Q0) {{\tiny $Q_0 =q_0$}};
	
	\node[squarednode] (Q1p1)[above right=0.6cm and 1cm of Q0]{{\tiny $Q_{1,1}$}};
	\node[squarednode] (Q1m1)[below right=0.6cm and 1cm of Q0]{{\tiny $Q_{2,1}$}};
	
	\node[squarednode] (Q1p1p1)[above right=0.1cm and 1cm of Q1p1]{{\tiny $Q_{1,2}$}};
	\node[squarednode] (Q1p1m1)[below right=0.1cm and 1cm of Q1p1]{{\tiny $Q_{2,2}$}};
	\node[squarednode] (Q1m1p1)[above right=0.1cm and 1cm of Q1m1]{{\tiny $Q_{3,2}$}};
	\node[squarednode] (Q1m1m1)[below right=0.1cm and 1cm of Q1m1]{{\tiny $Q_{4,2}$}};
  				
	\draw[-,draw=gray] (Q0.east) .. controls +(up:0mm) and +(down:0mm) ..  (Q1p1.west);
	\draw[-,draw=gray] (Q0.east) .. controls +(up:0mm) and +(down:0mm) ..  (Q1m1.west);	
	
	\draw[-,draw=gray] (Q1p1.east) .. controls +(up:0mm) and +(down:0mm) ..  (Q1p1p1.west);
	\draw[-,draw=gray] (Q1p1.east) .. controls +(up:0mm) and +(down:0mm) ..  (Q1p1m1.west);	
	
	\draw[-,draw=gray] (Q1m1.east) .. controls +(up:0mm) and +(down:0mm) ..  (Q1m1p1.west);
	\draw[-,draw=gray] (Q1m1.east) .. controls +(up:0mm) and +(down:0mm) ..  (Q1m1m1.west);			
	
	\node[text width=6cm] (T1)[above right=1.2cm and 5cm of Q0]
    {$\{q_0,Q_{1,1},Q_{1,2}\}$};
	\node[text width=6cm] (T2)[below=0.6cm of T1]
    {$\{q_0,Q_{1,1},Q_{2,2}\}$};
	\node[text width=6cm] (T3)[below=0.4cm of T2]
    {$\{q_0,Q_{2,1},Q_{3,2}\}$};
	\node[text width=6cm] [below=0.6cm of T3]
    {$\{q_0,Q_{2,1},Q_{4,2}\}$};
\end{tikzpicture}
\caption{Binomial tree diagram of the supply for $M=2$ time steps.}
\label{image:BTDiagram Q}
\end{figure}
\end{remark}

At time $t_k$, the discrete price process $\varpi$ takes the value $\varpi_k $, and the measurability condition w.r.t.
$\mathcal{F}_k$ means that $\varpi_k \in \left\{ \varpi_{1,k},\ldots,\varpi_{2^k,k}\right\}$, where the values $\varpi_{j,k}$ are unknown.
The discrete version of the optimal control problem \eqref{eq: Nagent min problem} reads

\begin{align}\label{eq: Discrete Nagent min problem}
\inf_{\substack{\mathbf{v}=(\mathbf{v}^1,\ldots,\mathbf{v}^N)\\ \mathrm{v}^i_k \in L^2_{\mathcal{F}_k}}}  & \frac{1}{N} \sum_{i=1}^N \left( \sum_{k=0}^{M-1} \frac{1}{2^k}\sum_{j=1}^{2^k} h\left(L(X_{j,k}^i,\mathrm{v}^i_{j,k}) + \varpi_{j,k} \mathrm{v}_{j,k}^i\right) + \frac{1}{2^M}\sum_{j=1}^{2^M} \Psi(X^i_{j,M}) \right) \nonumber
\\
\mbox{subject to} & \quad \frac{1}{N}\sum\limits_{i=1}^N \mathrm{v}_{j,k}^i=Q_{j,k} ~\mbox{ and } ~ X^i_{j,k}=X^i_{j,k-1} + h \mathrm{v}^i_{j,k-1} \nonumber
\\
& \quad \mbox{for }1 \leq j \leq 2^k, ~ 0\leq k\leq M-1, ~ 1\leq i \leq N.
\end{align}
\begin{remark}\label{Remark:UptoT}
Because we consider the Forward-Euler discretization of the stochastic processes $Q$ and $X$, the discrete approximation in \eqref{eq: Discrete Nagent min problem} of the integral \eqref{eq: Functional per agent} does not contain values at terminal time.
Moreover, since the terminal position $X^i_{j,M}$ is a function of %$(X^i_{j,k},v^i_{j,k})$ for $0 \leq k \leq M-1$
previous positions and velocities, the balance condition up to time-step $M-1$ ultimately determines the solution of \eqref{eq: Discrete Nagent min problem} up to time-step $M-1$; that is, the processes $\bv$ and $\varpi$ are not computed at terminal time $T$.
In contrast, the Hamilton-Jacobi approach adopted in Section \ref{Sec: LQ model} provides the values for both $\bv$ and $\varpi$ up to terminal time. Therefore, we consider the trajectories up to time step $M-1$.
\end{remark}

As in Section \ref{sec: NAgent problem}, we formulate a problem equivalent to \eqref{eq: Discrete Nagent min problem} for which the price corresponds to the Lagrange multiplier associated with the balance condition.
Using the discrete balance condition in \eqref{eq: Discrete Nagent min problem}, we write
\begin{gather*}
\frac{1}{N} \sum_{i=1}^N \sum_{k=0}^{M-1} \sum_{j=1}^{2^k} \frac{1}{2^k} h\varpi_{j,k} \mathrm{v}_{j,k}^i=\sum_{k=0}^{M-1} \sum_{j=1}^{2^k}  \frac{1}{2^k} h\varpi_{j,k} Q_{j,k}.
\end{gather*}
Replacing the left-hand side of the previous equation in the functional to minimize in \eqref{eq: Discrete Nagent min problem}, we get 
\[
	\frac{1}{N} \sum_{i=1}^N \left( \sum_{k=0}^{M-1} \frac{1}{2^k}\sum_{j=1}^{2^k} h L(X_{j,k}^i,\mathrm{v}^i_{j,k})  + \frac{1}{2^M}\sum_{j=1}^{2^M} \Psi(X^i_{j,M}) \right) +\sum_{k=0}^{M-1} \sum_{j=1}^{2^k}  \frac{1}{2^k} h\varpi_{j,k} Q_{j,k},
\]
where the last term is independent of $\mathbf{v}$.
Hence, we consider the equivalent discrete minimization problem
\begin{gather}\label{eq:NQ discrete Nagent min no price}
\inf_{\substack{\mathbf{v}=(\mathbf{v}^1,\ldots,\mathbf{v}^N)\\ \mathrm{v}^i_k \in L^2_{\mathcal{F}_k}}} \frac{1}{N} \sum_{i=1}^N \left( \sum_{k=0}^{M-1} \frac{1}{2^k} \sum_{j=1}^{2^k}  hL(X_{j,k}^i,\mathrm{v}^i_{j,k}) + \frac{1}{2^M} \sum_{j=1}^{2^M} \Psi(X^i_{j,M})\right)
\\
\mbox{subject to }~   g_{j,k}(\mathbf{v})=0, ~ X^i_{j,k+1}=X^i_{j,k} + h \mathrm{v}_{j,k}^i ,   ~ \mbox{for }1 \leq j \leq 2^k, ~ 0\leq k\leq M-1,
\nonumber
\end{gather}
where
\begin{equation}\label{def:NQ aux g}
	g_{j,k}(\mathbf{v}) := \frac{1}{N}\sum\limits_{i=1}^N \mathrm{v}_{j,k}^i-Q_{j,k}, \quad \mbox{for }1 \leq j \leq 2^k, ~ 0\leq k\leq M-1.
\end{equation}
To solve this minimization problem with equality constraints, we consider the augmented Lagrangian 
\begin{equation}\label{eq: Augmented L Nagent}
	\tilde{L}(\mathbf{v},\bm{\lambda})=\frac{1}{N} \sum_{i=1}^N \left( \sum_{k=0}^{M-1} \frac{1}{2^k}\sum_{j=1}^{2^k} hL(X_{j,k}^i,\mathrm{v}^i_{j,k}) + \frac{1}{2^M} \sum_{j=1}^{2^M} \Psi(X^i_{j,M})\right) + \sum_{k=0}^{M-1} \sum_{j=1}^{2^k} \lambda_{j,k} g_{j,k}(\mathbf{v}),
\end{equation}
where $\bm{\lambda}$ is a vector with components $\lambda_{j,k}\in \Rr$, for $j=1,\ldots,2^k$ and $k=0,\ldots,M-1$.
If the functions $g_{j,k}$ are convex, any minimizer $\mathbf{v}$ of \eqref{eq:NQ discrete Nagent min no price} is characterized by the existence of a multiplier $\bm{\lambda}$ such that $(\mathbf{v},\bm{\lambda})$ solves the Karush-Kuhn-Tucker condition (\cite{BoydConvex}, Section 5.5.3)
\begin{gather}\label{eq:NQ KKT condition}
D_{\mathbf{v}} \left(	\frac{1}{N} \sum_{i=1}^N \left( \sum_{k=0}^{M-1} \frac{1}{2^k}\sum_{j=1}^{2^k} hL(X_{j,k}^i,\mathrm{v}^i_{j,k}) + \frac{1}{2^M} \sum_{j=1}^{2^M} \Psi(X^i_{j,M})\right)\right) 
\\
+ \sum_{k=0}^{M-1} \sum_{j=1}^{2^k} \lambda_{j,k} D_{\mathbf{v}} g_{j,k}(\mathbf{v}) =0, \nonumber
\end{gather}
where $D_{\mathbf{v}}$ denotes the gradient w.r.t.
the variables $v^i_{j,k}$ for $i=1,\ldots,N$, $k=0,\ldots,M-1$, and $j=1,\ldots,2^k$.
In turn, any solution $(\mathbf{v},\bm{\lambda})$ of \eqref{eq:NQ KKT condition} defines a price process.
To see this, we use the definition of $g_{j,k}$ in \eqref{def:NQ aux g} to write the last term in \eqref{eq: Augmented L Nagent} as
\begin{align}\label{eq:Aux NQ g term}
	& \sum_{k=0}^{M-1} \sum_{j=1}^{2^k} \frac{1}{2^k}\left(2^k\lambda_{j,k}\right) g_{j,k}(\mathbf{v}) \nonumber
	\\
	&= \frac{1}{N}\sum\limits_{i=1}^N \sum_{k=0}^{M-1} \sum_{j=1}^{2^k} \frac{1}{2^k}\left(2^k\lambda_{j,k} \right) \mathrm{v}_{j,k}^i- \sum_{k=0}^{M-1} \sum_{j=1}^{2^k} \lambda_{j,k} Q_{j,k}.
\end{align}
Notice that the last term on the right-hand side of \eqref{eq:Aux NQ g term} is independent of $\mathbf{v}$.
Therefore, any minimizer of the functional
\begin{align*}
\bv \mapsto & \frac{1}{N} \sum_{i=1}^N \bigg( \sum_{k=0}^{M-1} \frac{1}{2^k}\sum_{j=1}^{2^k} h\left(L(X_{j,k}^i,\mathrm{v}^i_{j,k}) + 2^k\lambda_{j,k} \mathrm{v}_{j,k}^i\right) 
+ \frac{1}{2^M}\sum_{j=1}^{2^M} \Psi(X^i_{j,M}) \bigg) 
\\
& - \sum_{k=0}^{M-1} \sum_{j=1}^{2^k} \lambda_{j,k} Q_{j,k},
\end{align*}
subject to the constraints 
\begin{equation}\label{eq:DiscreteConstrains}
\frac{1}{N}\sum\limits_{i=1}^N \mathrm{v}_{j,k}^i=Q_{j,k} ~\mbox{ and } ~ X^i_{j,k+1}=X^i_{j,k} + h \mathrm{v}_{j,k}^i,
\end{equation}
for $1 \leq j \leq 2^k$, $0\leq k\leq M-1$, and $1\leq i \leq N$, is also a minimizer of the problem
\begin{align*}
\inf_{\substack{\mathbf{v}=(\mathbf{v}^1,\ldots,\mathbf{v}^N)\\ \mathrm{v}^i_k \in L^2_{\mathcal{F}_k}}}  &  \frac{1}{N} \sum_{i=1}^N \left( \sum_{k=0}^{M-1} \frac{1}{2^k}\sum_{j=1}^{2^k} h\left(L(X_{j,k}^i,\mathrm{v}^i_{j,k}) + 2^k\lambda_{j,k} \mathrm{v}_{j,k}^i\right) + \frac{1}{2^M}\sum_{j=1}^{2^M} \Psi(X^i_{j,M}) \right) 
\end{align*}
subject to \eqref{eq:DiscreteConstrains}, which corresponds to \eqref{eq: Discrete Nagent min problem} when 
\begin{equation}\label{eq:BTprice Lagrange multiplier formula}
\varpi_{j,k}:=2^k\lambda_{j,k} \quad \mbox{for}~ 1 \leq j \leq 2^k,~ 0\leq k\leq M-1.
\end{equation}
Hence, the minimizer $\mathbf{v}$ of \eqref{eq:NQ discrete Nagent min no price} and $\varpi$, as defined before, solve \eqref{eq: Discrete Nagent min problem}.

\subsection{Numerical tests for the linear-quadratic case}

Here, we implement the previous scheme on the model of Section \ref{Sec: LQ model}, and we illustrate the convergence as the number of players increases.

We assume that the supply $Q$ follows the linear dynamics \eqref{eq:Q SDE mean reverting}, where $\overline{Q}(t) = \sin (2 \pi t)$, $\sigma_s = 0.05$, and $Q_0 = 0.1$. For $N\in \Nn$, the initial values $x_0^1,\ldots,x_0^N$ for the state of the agents are sampled from a normal distribution with mean $0$ and standard deviation $0.1$, which corresponds to $\tilde{m}_0 \sim \mathcal{N}(0,0.1)$ in the continuous model. We refer to the price given by \eqref{eq:BTprice Lagrange multiplier formula}, where $\bm{\lambda}$ is the solution of \eqref{eq:NQ KKT condition}, as $\varpi^N$. The price computed using the Forward-Euler discretization of \eqref{eq: LQ Augmented dynamics} is denoted by $\varpi^\infty$, and it is computed as
\begin{equation}\label{eq:ForwardEulerPriceInfty}
\varpi_{k+1}^\infty=\varpi_{k}^\infty+b^P(X_k,\overline{X}_k ,Q_k,\varpi_k)h+\sigma^P(X_k,\overline{X}_k ,Q_k,\varpi_k)\Delta W_k, 
\end{equation}
$k=0,\ldots,M-1$, where $b^P$ and $\sigma^P$ are given by \eqref{eq: Q Coeffiecients relation price supply} and $\varpi_0$ is given by \eqref{eq:Initial price LQ}. We take $M=11$ time steps, so $h=0.09$. The remaining parameters are selected as follows
\[
	T=1,\; \eta=c=1,\; \kappa=\zeta=0.25,\; \gamma=e^2.
\]
%First, we recall that \eqref{eq:Initial price LQ} provides an exact expression for the initial value of the price.
%Hence, we evaluate the convergence of the initial value for the price obtained by the discrete model against the exact value \eqref{eq:Initial price LQ}. {\color{red} This convergence only makes sense when $N\to \infty$}. 

%Second, we compare the discrete paths of the price computed using the Binomial Tree with those computed using the Hamiton-Jacobi equation. {\color{red} This convergence only makes sense when $N\to \infty$}. 

To illustrate the convergence as $N$ increases, we compute the mean discrete $L^2$ difference
\[
	\overline{\|\varpi^N-\varpi^\infty\|}_{L^2}=\frac{1}{2^{M-1}} \sum_{j=1}^{2^{M-1}} \|\varpi(j)^N - \varpi(j)^\infty\|_{L^2},
\]
where $j$ denotes the realization of the supply for which $\varpi(j)^N$ and $\varpi(j)^\infty$ approximate $\varpi^N$ and $\varpi^\infty$, respectively. This guarantees that the comparison between the trajectories relies on the same source of noise. Thus, recalling that the increments for the Binomial Tree are $\pm \sqrt{h}$, we take the same increments in the discretization of \eqref{eq:System price LQ}. Therefore, the supply in \eqref{eq:Qdiscrete Forward Euler} is the same for both $\varpi^N$ and $\varpi^\infty$. Following Remark \ref{Remark:UptoT}, we consider each path up to time-step $M-1$.

As shown in Table \ref{table:Potential}, $\overline{\|\varpi^N-\varpi^\infty\|}_{L^2}$ decreases as the number of players increases, which in turn corresponds to $\overline{x}_0$ converging to $\mu_0=0$. Figure \ref{image:Sampleeta1} shows all possible paths of the price, up to time-step $M-1$, for the two discrete approximations as $N$ varies. We notice that the convergence of $\varpi^N$ to $\varpi^\infty$ strongly depends on the convergence of the initial value at $t=0$, which is a consequence of the necessary condition $\overline{x}_0 \to \mu_0$ as $N\to \infty$. For some trajectories, we observe negative prices due to market flooding. This behavior has been observed in crude oil futures prices during pandemic times, as the West Texas Intermediate (WTI) crude oil price dropped to negative levels during April 2020, ending at minus $\$37.63$ a barrel. It is possible to elaborate on the computation of market flooding times by studying the first hitting time of the representation \eqref{eq: LQ Augmented dynamics} when $\varpi^\infty$ becomes negative. Figure \ref{image:Samplepatheta1} shows four sample paths of the supply and the corresponding prices $\varpi^N$ (for $N=50$) and $\varpi^\infty$. We observe a negative correlation between supply and price, verified by the covariance between supply and price illustrated in Figure \ref{image:Cov w infty}.

\begin{remark}
Because we approximate $\varpi^\infty$ using a step size $h$, the convergence of the forward scheme \eqref{eq:ForwardEulerPriceInfty} is guaranteed as $h\to \infty$. On the other hand, for $\varpi^N$, it is possible to consider not only the convergence as $N\to \infty$ but also the convergence as $h \to 0$. The former relates to the convergence of a finite game to a continuum (MFG) game. The latter relates to the convergence of the discrete version of noise to its continuous counterpart, which depends on $h\to 0$. 
Regarding the computation of $\varpi^N$, notice that adding one player to a scheme with $M$ time steps requires $2^{M+1}-1$ additional variables. On the other hand, increasing by one the number of time steps for $N$ players requires $ (N+1)2^{M+1}$ additional variables. For this reason, we fixed the number of time steps to be $M=11$ in the previous test and illustrated only the convergence as $N$ increases.
\end{remark}

\begin{remark}
In the large-time behavior of the mean-reverting dynamics \eqref{eq:Q SDE mean reverting}, $Q$ asymptotically  approaches the equilibrium $\overline{Q}$. However, we do not observe the large-time behavior in our simulations because we consider it a finite time horizon problem.
\end{remark}

%The number of variables used to compute $\varpi^N$ is 
%\[
%	(N+1)\sum_{k=0}^{M-1} 2^k,
%\]
%which increases exponentially as the number of time steps increases.

\begin{table}[hbt!]
\renewcommand{\arraystretch}{1.2}
\begin{tabular}{ |c|c|c|c| } 
 \hline
& $N=10$ & $N=30$ & $N=50$\\ 
 \hline
$\left|\overline{x}_0-\mu_0\right|$ & $4.48395*10^{-2}$ &  $5.84296*10^{-3}$ & $5.54493*10^{-4}$\\
$\overline{\|\varpi^N-\varpi^\infty\|}_{L^2}$ & $8.94968*10^{-1}$ &  $4.25748*10^{-1}$ & $2.59851*10^{-1}$\\
Variables ($\varpi^N$) & $22517$ & $63457$ & $104397$ \\
 \hline
\end{tabular}
\caption{Convergence of the initial mean position (first row) and the price processes (second row). Number of variables of the Binomial Tree implementation to compute $\varpi^N$ (third row).}
\label{table:Potential}
\end{table}

\begin{figure}[ht]
\centering
\includegraphics[width=0.4\textwidth]{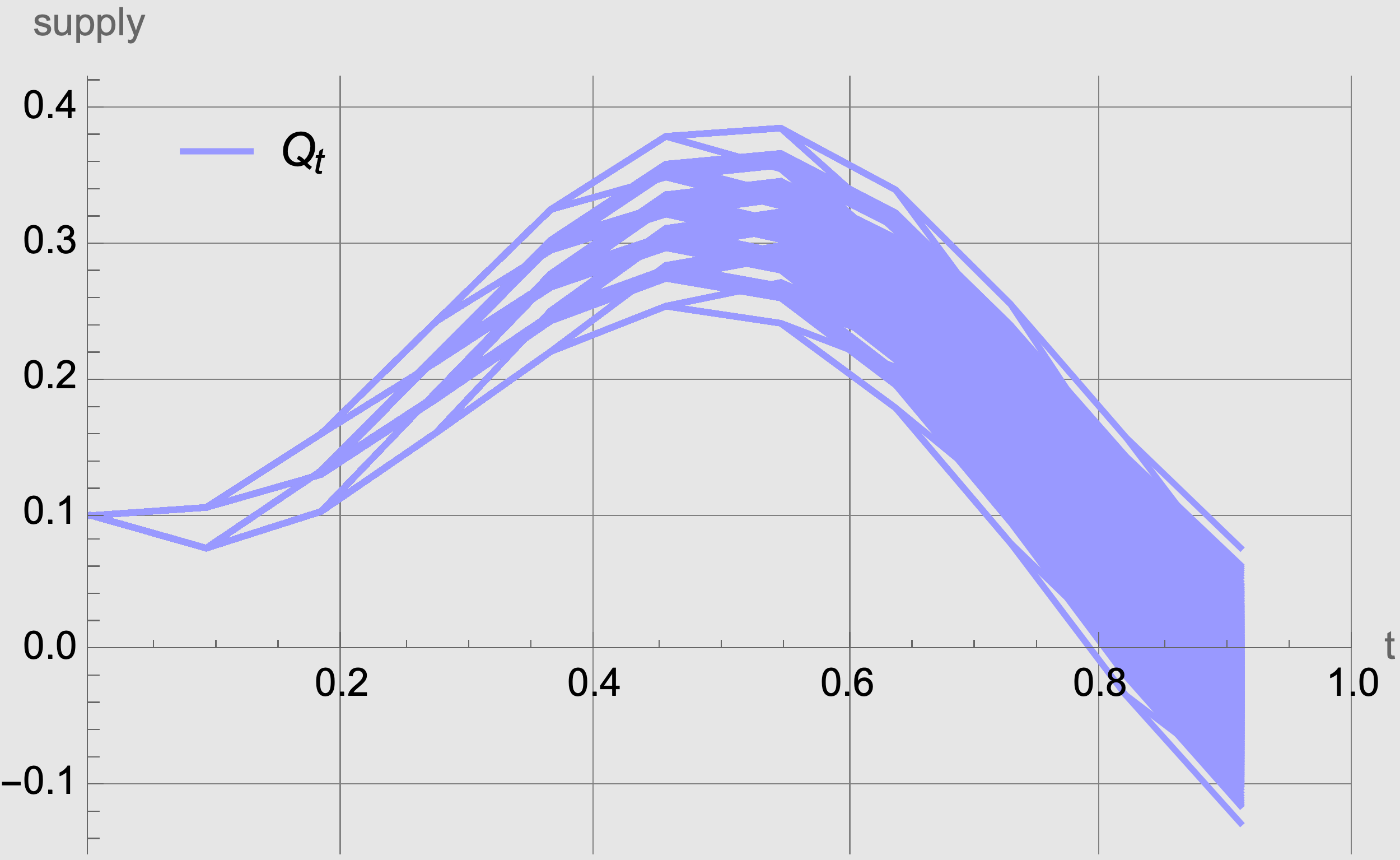}
\includegraphics[width=0.4\textwidth]{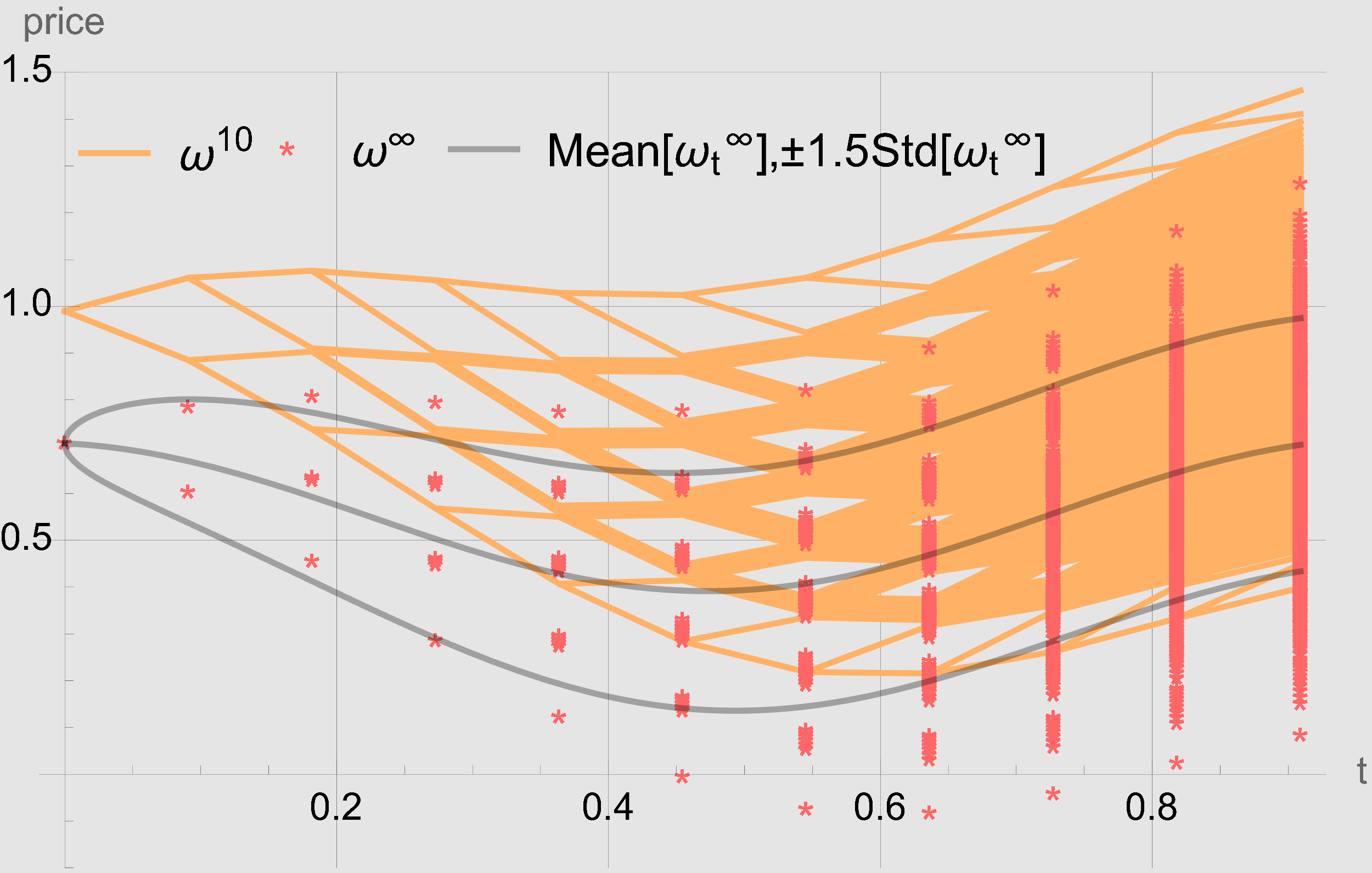}
\includegraphics[width=0.4\textwidth]{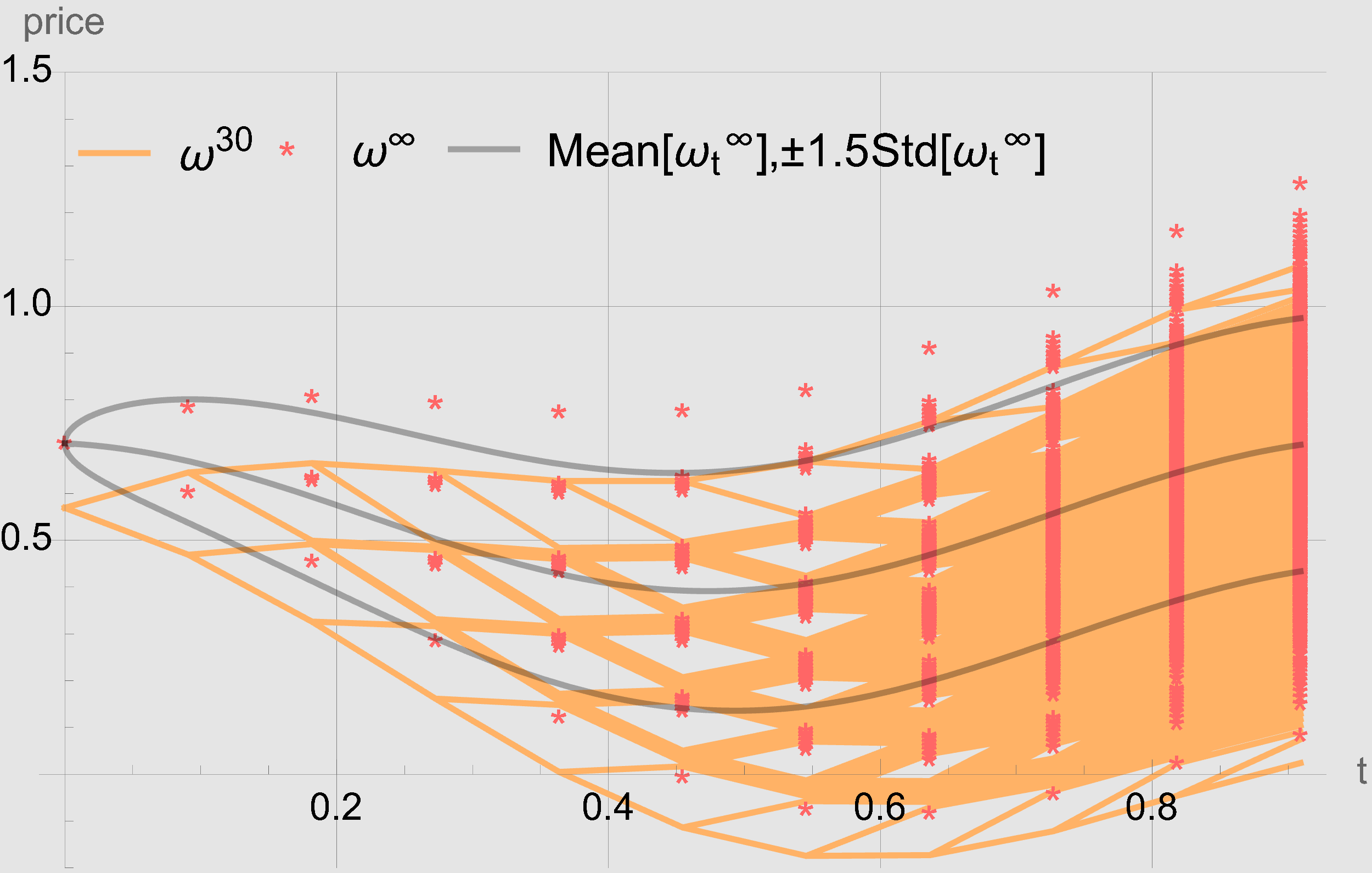}
\includegraphics[width=0.4\textwidth]{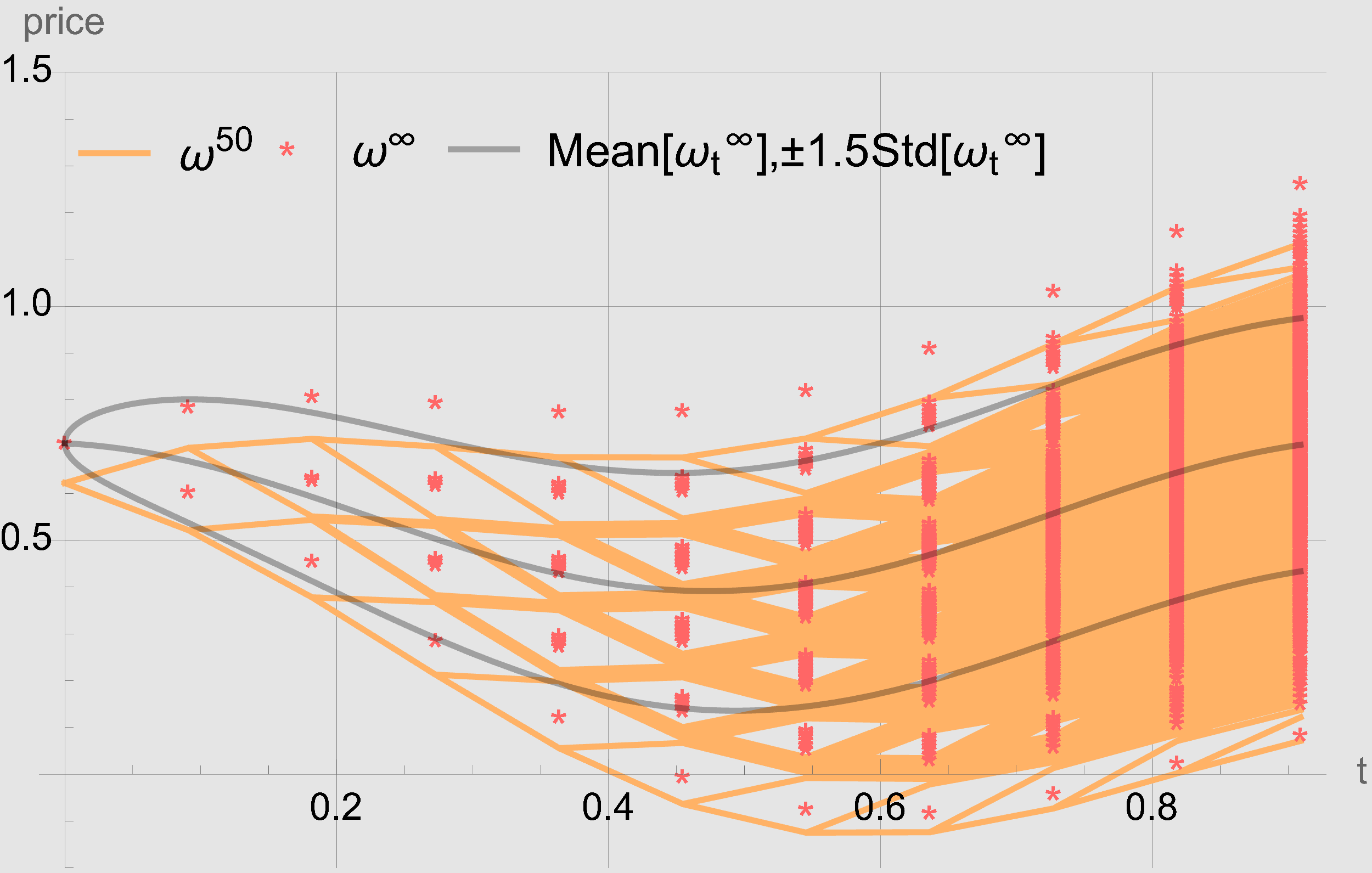}
\caption{(Top-left) Binomial tree supply and prices $\varpi^N$ and $\varpi^\infty$ for $N \in \{10,30,50\}$. Statisticis of $\varpi^\infty$ (gray curves).}
\label{image:Sampleeta1}
\end{figure}

\begin{figure}[ht]
\centering
\includegraphics[width=0.4\textwidth]{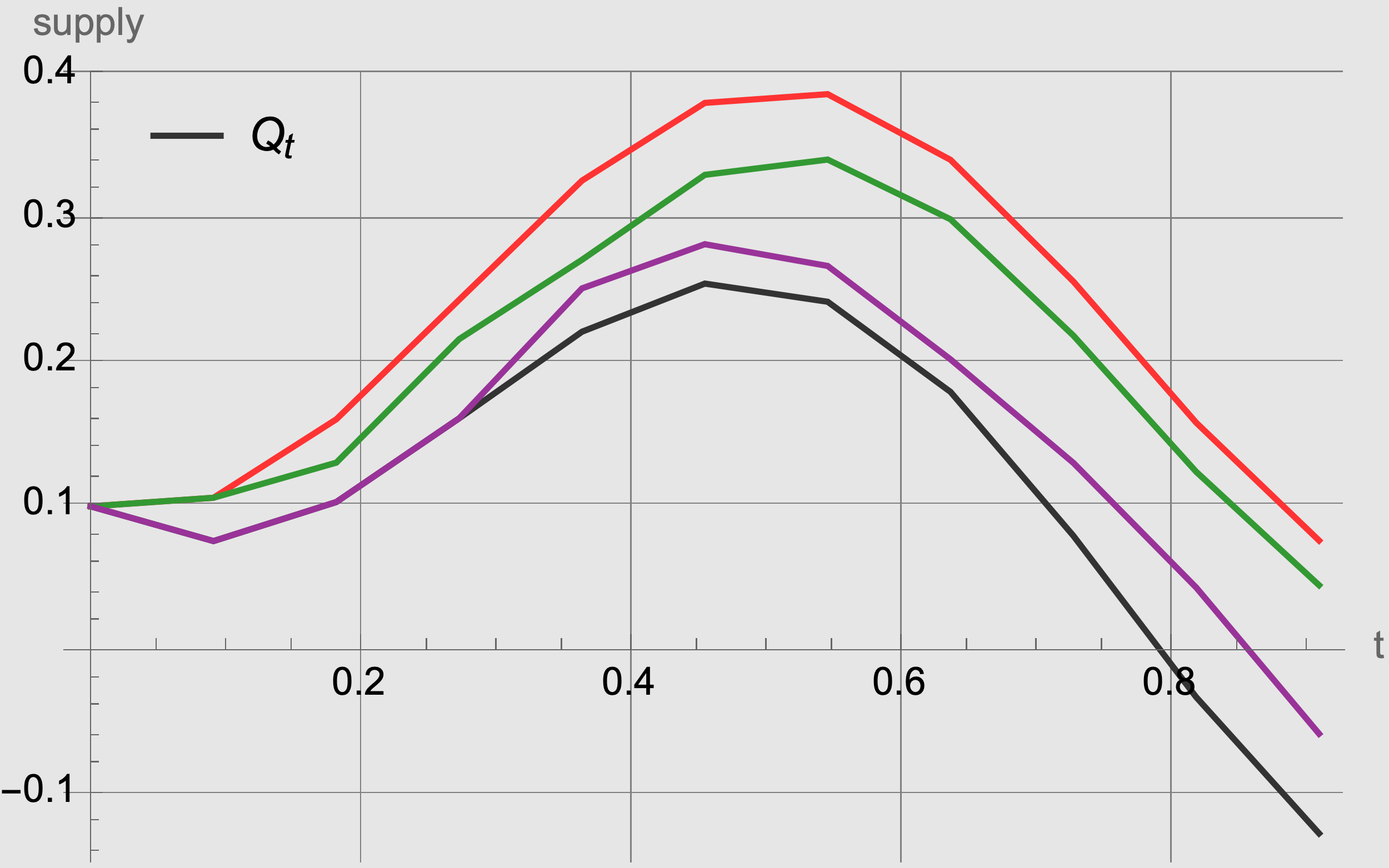}
\includegraphics[width=0.4\textwidth]{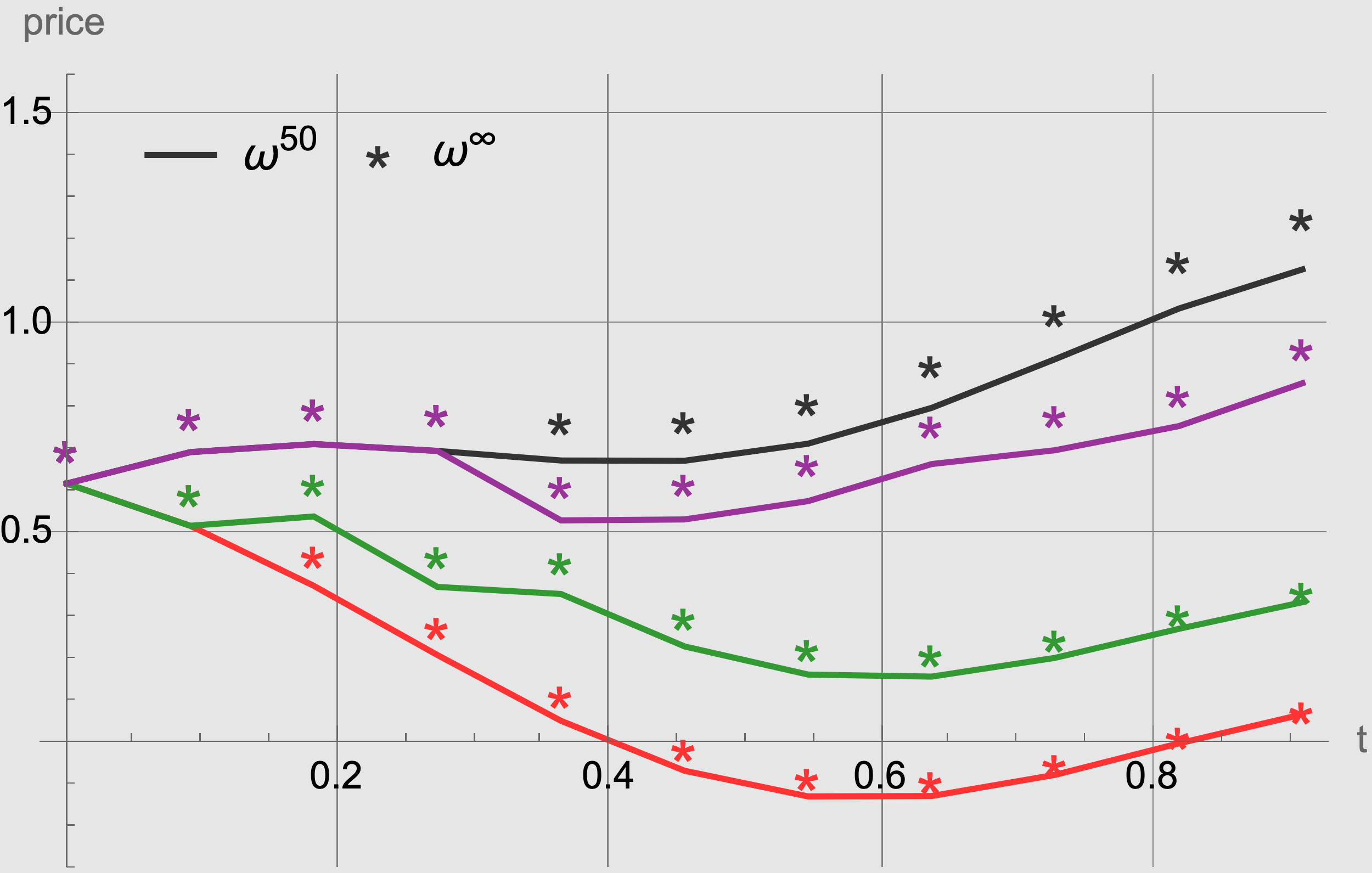}
\caption{(Left) Sample trajectories of the supply. (Right) Corresponding prices $\varpi^N$, for $N=50$, and $\varpi^\infty$ (right).}
\label{image:Samplepatheta1}
\end{figure}

\begin{figure}[ht]
\centering
\includegraphics[width=0.4\textwidth]{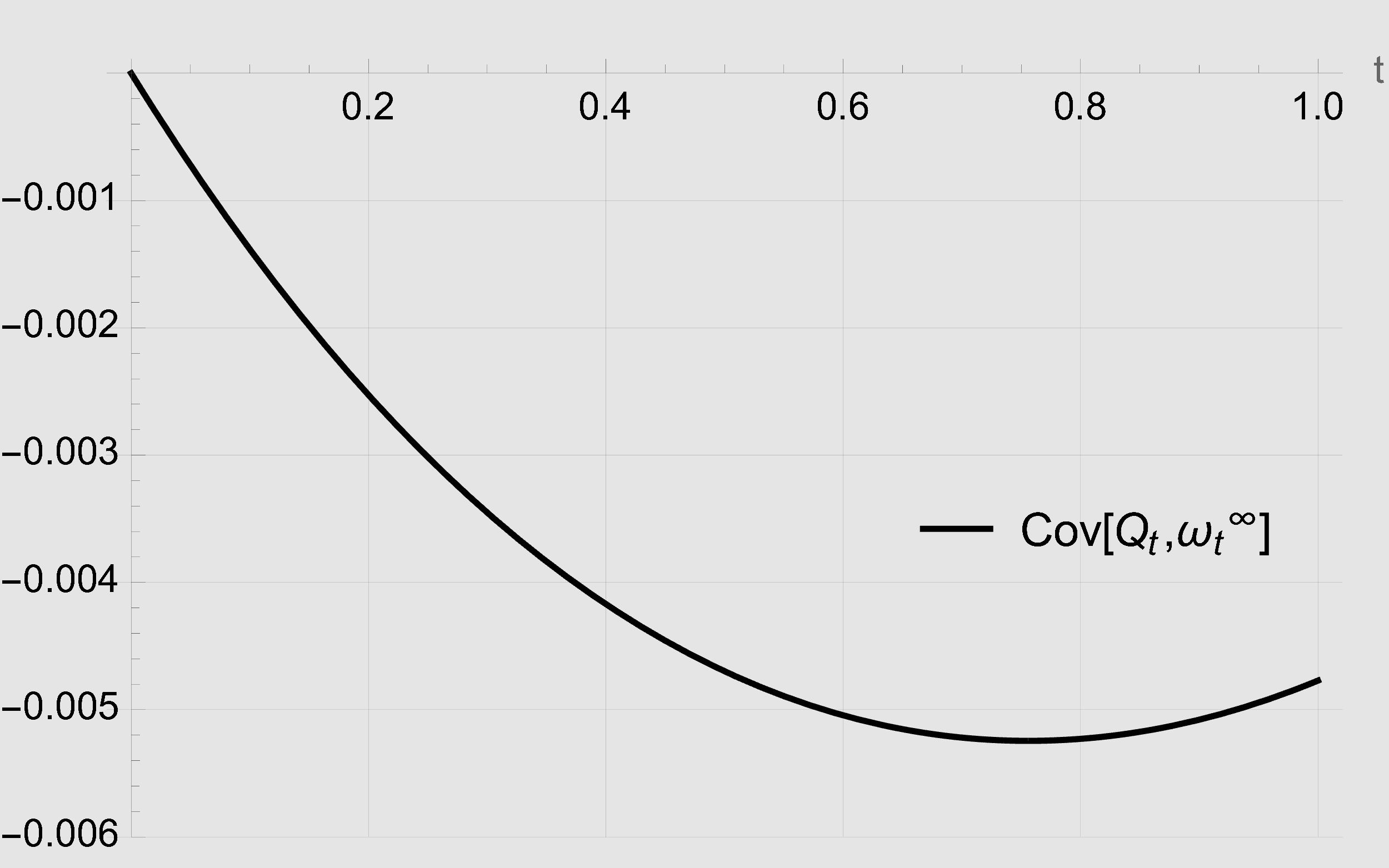}
\caption{Covariance between $Q$ and $\varpi^\infty$.}
\label{image:Cov w infty}
\end{figure}

\subsection{Real data test}
Here, we parametrize the linear-quadratic model of Section \ref{Sec: LQ model} using real data from the electric grid in Spain. Using the parameterized model, we illustrate the prices obtained from the continuum game.

We use the data of consumption and price from the market in Spain. The data is available at the website \href{https://www.esios.ree.es}{https://www.esios.ree.es}. We use the hourly demand (megawatts) for the working days of March 2022, so $T=24$ hours. Recall that in our model, both the instantaneous supply $Q$ and the agents provide electricity to the grid, as we assume that each agent has a device storing $X_t$ units of electricity at time $t$, which can be further stored or traded in the market. On the contrary, in the electricity grid represented by the real data, agents consume electricity, and no interaction with the market takes place. Therefore, the supply $Q$ we take for our model corresponds to minus the demand observed in the data. 

First, we parametrize the supply function. To do so, we assume it is given by
\begin{equation}\label{eq:RealD Q def}
	Q_t = Q_{osc}(t)+Q^{W}_t,
\end{equation}
where $Q_{osc}:[0,T]\to \Rr$ and 
\begin{equation}\label{eq:RealD Vasicek}
	dQ^{W}_t = \theta\left(\overline{Q}-Q^{W}_t\right)dt + \sigma_s dW_t,
\end{equation}	 
for some $\theta,\overline{Q},\sigma_s \in \Rr$. Therefore, $Q$ follows the linear dynamics \eqref{eq:Q linear dynamics} for
\[
	b_0^S(t) = \dot{Q}_{osc}(t) + \theta\left(\overline{Q}-Q_{osc}(t)\right), \quad b_1^S(t) = -\theta, \quad \sigma_0^S(t) = \sigma_s, \quad \sigma_1^S(t) =0.
\]
We fit $Q_{osc}$ using the mean supply of the data set. Assuming that $Q_{osc}$ is a linear combination of sines and cosines, we obtain
\begin{align*}
	Q_{osc}(t)=& 0.883118 \sin (2 \pi  t)+0.675294 \sin (4 \pi  t)+0.190316 \sin
   (6 \pi  t)+0.0248343 \sin (8 \pi  t)
   \\
   & +0.750615 \cos (2 \pi 
   t)-0.25301 \cos (4 \pi  t)-0.0233308 \cos (6 \pi  t)+0.191395
   \cos (8 \pi  t)
   \\
   & -0.027736.
\end{align*}
Because the left-hand side of \eqref{eq:RealD Q def} corresponds to the observed data, we fit the parameters $\theta$, $\overline{Q}$, and $\sigma_s$ using the maximum-likelihood estimator of \eqref{eq:RealD Vasicek} (see \cite{BrigoMercurio2006}, Chapter 3) with time step $h=1/23=0.0434783$. We obtain
\[
	\theta = 35.9957,\quad \overline{Q}=-0.0186653, \quad \sigma_s= 0.860584.
\]
For the initial value of the supply, we take $q_0=Q_{osc}(0)+Q^W_0$, where $Q^W_0$ is the mean of the observed differences $Q_0 - Q_{osc}(0)$. Figure \ref{image:RealD Q price show} depicts the (normalized) supply data and the parameterized supply function. Next, we fit the parameters of the cost functions in \eqref{def:L Psi LQ}. To do so, we use the expression for the deterministic linear-quadratic model in \cite{gomes2018mean}. In this setting, the MFG price is
\[
	\varpi(t)=\eta \left( \kappa - \mu_0 \right) (T-t) + \gamma \left(\zeta-\mu_0 \right) - \eta \int_t^T \int_0^s Q(r)dr \;ds - \gamma \int_0^T Q(s)ds - c Q(t).
\] 
We take $Q=Q_{osc}$ in the previous expression, and we fit the parameters using the mean price of the data and least-squares. We obtain
\begin{align*}
& \eta = 0.00176489, \; \kappa = -371.936, \; c = 0.472603,  
\\
& \gamma  = 0.000877786, \; \zeta =  377.536, \; \mu_0 = 1.74687.
\end{align*}
Then, we can compare the observed price data with the corresponding trajectory of the price $\varpi^\infty$ obtained in \eqref{eq:System price LQ}. Given a supply trajectory from the data $Q^j$, we use \eqref{eq:RealD Q def}, \eqref{eq:RealD Vasicek}, and \eqref{eq:Qdiscrete Forward Euler} to compute the corresponding noise trajectory  $\Delta W_k^j$,
\[
	\Delta W^j_k = \frac{Q_{k+1}^j-Q_{k}^j - h b^S(Q_k^j,k)}{\sigma^S (Q_k^j , k)}, \quad k=0,\ldots,23,
\]
which we use in \eqref{eq:System price LQ}. Figure \ref{image:RealD Q price show} depicts three price trajectories. 

\begin{figure}[ht]
\centering
\includegraphics[width=0.49\textwidth]{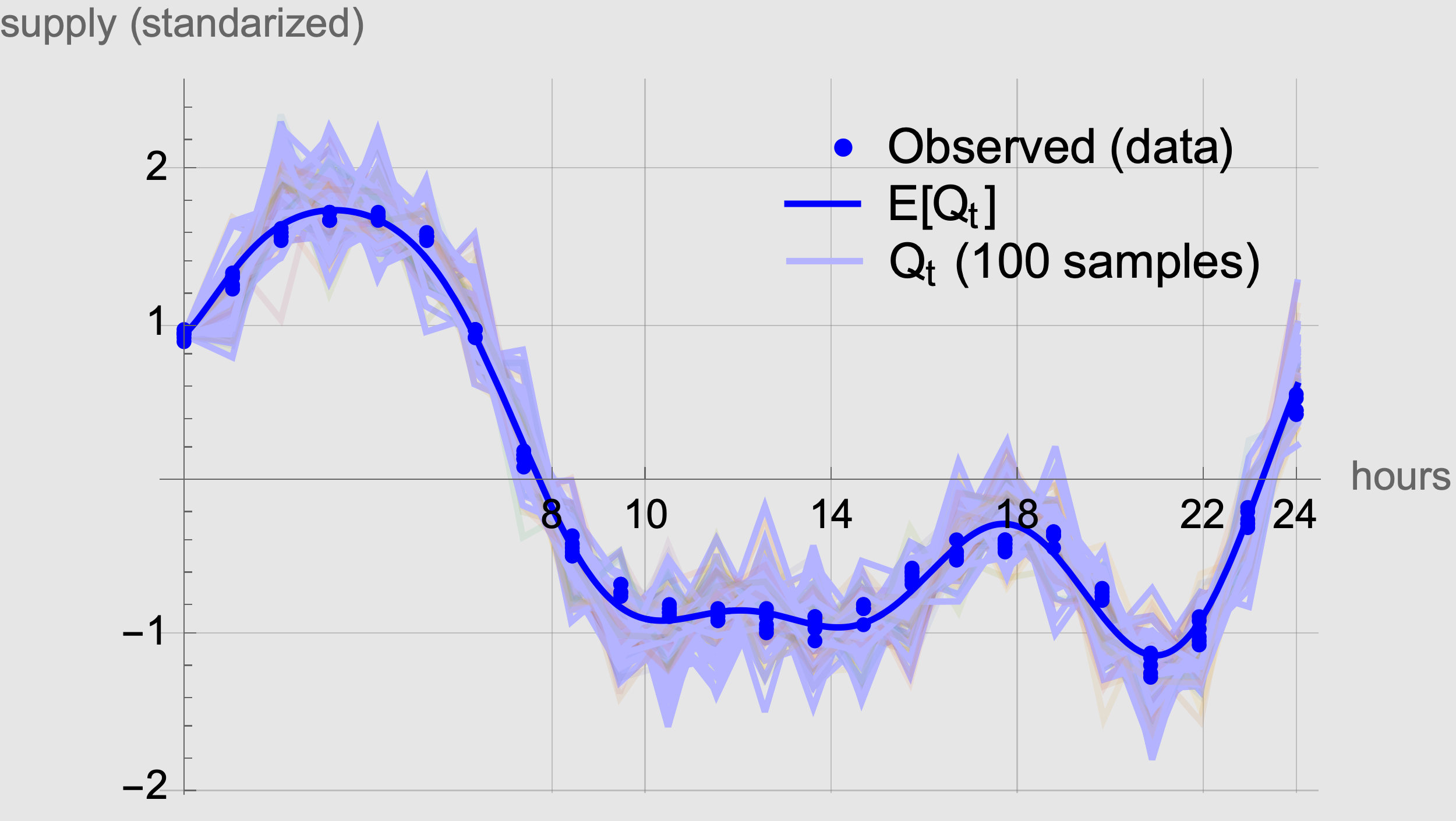}
\hskip0.3cm
\includegraphics[width=0.45\textwidth]{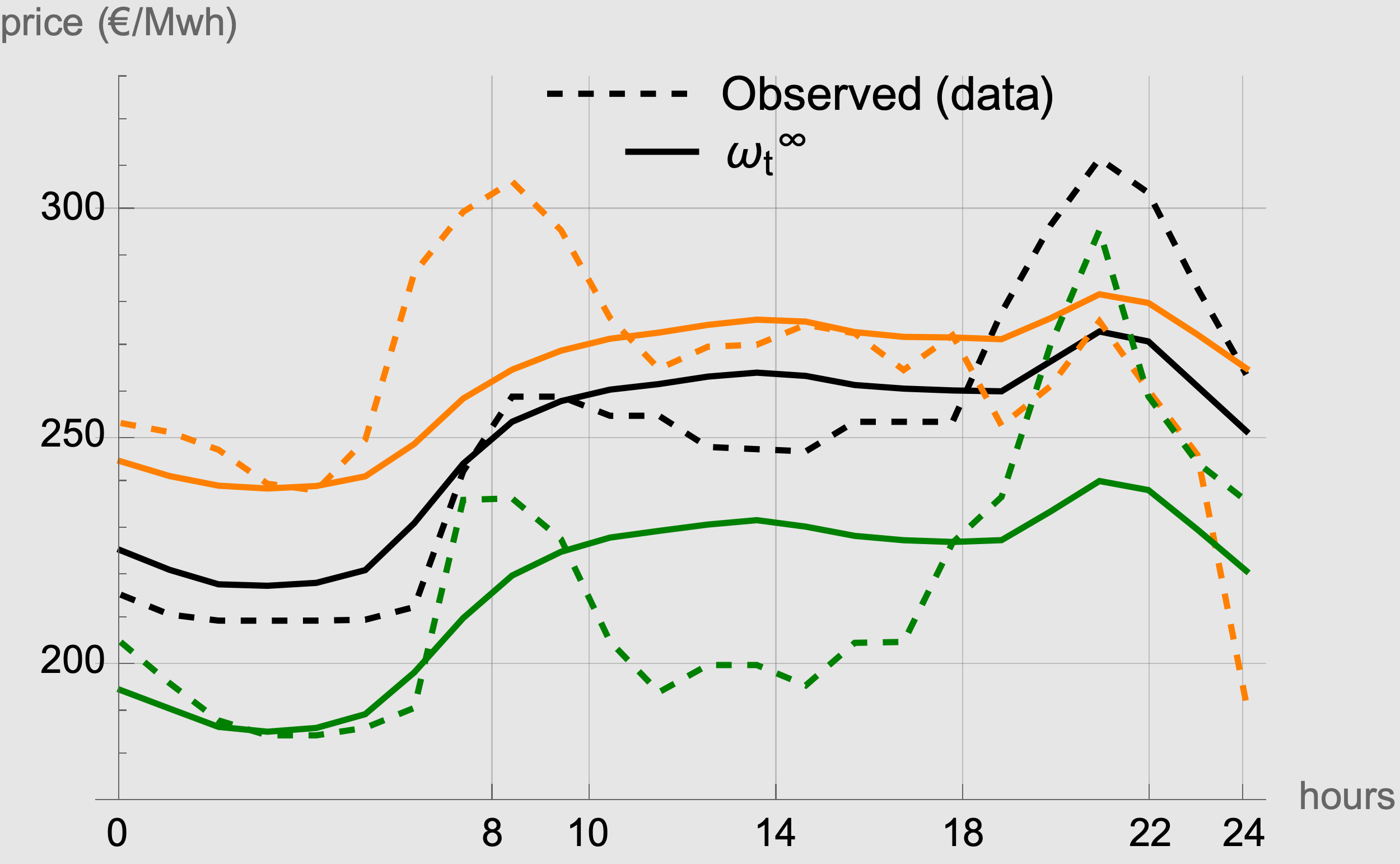}
\caption{(Left) Observed supply (data) and simulated supply. (Right) Observed price (3 data samples) and corresponding simulated price $\varpi^\infty$.}
\label{image:RealD Q price show}
\end{figure}

We observe that price peaks are smoothed, and price variations are reduced. Thus, the price formation mechanism dumps the volatility effect coming from the supply side, and the market may benefit from the smoothing effect. For instance, in June 2021, the Spanish electric introduced voluntary prices for small consumers. The tariffs distinguish three regimes: The peak period (10-14 hrs, 18-22 hrs), the flat period (8-10 hrs, 14-18 hrs, 22-24 hrs), and the valley period (24-8 hrs). The prices are published for the following day, so consumers can decide when to consume energy. If this policy is implemented on a big scale, our price formation model will provide an alternative to balance the different tariffs across regimes.

\section{Conclusions and further directions}
A price formation model for a finite number of agents is presented. This model corresponds to the particle approximation of the continuum model introduced in \cite{gomes2018mean}. Under convexity and growth assumptions on the cost functions, we proved the solvability of Problem \ref{problem: problem}. We presented an approach for the numerical solution of the model with a continuum population and another approach for the finite population model. 

The approach for the numerical solution of the continuum game uses the Hamilton-Jacobi equation that corresponds to the stochastic optimal control problem that each agent solves. In this case, we characterize the price as the solution of an SDE, whose initial condition (the price value at initial time) admits an explicit expression.  Therefore, the error in the approximation depends only on the discrete scheme used to approximate the solution of such SDE. In particular, we use a Forward-Euler scheme, for which the error depends on the time-step size, which can be arbitrarily small without high computational cost due to the explicit nature of the forward scheme. This approach is developed for the linear-quadratic structure of the supply and cost functions.

The approach for the numerical solution of the finite game is suited for any convex cost structure and any supply dynamics. Here, we implement it for the linear-quadratic case only. It relies on the binomial tree approximation of the noise present in the SDE for the supply. As a result, the price is characterized as the Lagrange multiplier of a high-dimensional convex optimization problem with constraints. In this case, as the time-step size decreases, the number of variables in the optimization problem grows exponentially. Therefore, we can not overcome the curse of dimensionality in implementing this approach. However, the results are in good agreement with the theoretical ones. 
%The discrepancies in the path trajectories decrease at a high computational cost. Hence, increasing the number of time-steps is not feasible. Thus, we need a new approach to provide a practical solution to these problems.

The qualitative properties of the price obtained by our schemes agree with what is observed in several markets. Fluctuations in the supply are negatively correlated with the price. For the linear-quadratic setting, two relations are observed in the drift of the price: increasing the running trading rate costs $c$ forces the price to move opposite to the supply dynamics, and the price increases when the time-average supply exceeds the preferred running state $\kappa$ of the agents. Moreover, because essentially, the drift determines the expected value of the price, and the volatility determines its variability, we see that the relation between the time-average supply and the preferred state of the agents determines the mean price, while increments on the trading cost increase the variability of the price. Finally, our model provides the scenario for which market saturation results in negative prices.

Other approaches, such as Machine Learning, can be implemented to deal with the high-dimensional nature of Problem \ref{problem: problem} as the number of players increases.

In our model, the supply of the commodity is an exogenous process; that is, the supply is an input quantity for the model. A further extension is to consider a supply that depends on the price. In this case, both supply and price would be endogenous variables for the model, and they would be determined by the optimal interaction of agents with the market. 

%----------------------------------------------------------------------------------------------------------------------------------
\bibliographystyle{plain}
% Diogo
%\IfFileExists{"/Users/gomesd/mfgDGOFFICE.bib"}
%{\bibliography{/Users/gomesd/mfgDGOFFICE.bib}}

% Ricardo
%\IfFileExists{"/Users/ribeirrd/Dropbox/MasterBIB/mfg.bib"}
%{\bibliography{/Users/ribeirrd/Dropbox/MasterBIB/mfg.bib}}
{\bibliography{mfg.bib}} %fix references in the master bib!
\end{document}